\documentclass[11pt]{article}

\usepackage[linesnumbered,ruled]{algorithm2e}

\usepackage{jhtitle}
\usepackage{geometry}
\usepackage{amsbsy,amsmath,amsthm,amssymb,graphicx}
\usepackage{natbib}
\usepackage{hyperref}
\usepackage{enumitem}
\usepackage{pstricks}
\usepackage{subcaption}

\newcommand{\pcite}[1]{\citeauthor{#1}'s \citeyearpar{#1}}

\newcommand{\df}{\mathrm{d}}

\newtheorem{theorem}{Theorem}
\newtheorem{lemma}[theorem]{Lemma}

\newtheorem{remark}[theorem]{Remark}
\newtheorem{proposition}[theorem]{Proposition}

\newcommand{\X}{{\mathsf{X}}}

\newcommand{\by}{{\boldsymbol{y}}}
\newcommand{\bx}{{\boldsymbol{x}}}
\newcommand{\bth}{{\boldsymbol{\theta}}}
\newcommand{\bet}{{\boldsymbol{\eta}}}
\newcommand{\bal}{{\boldsymbol{\alpha}}}
\newcommand{\bbe}{{\boldsymbol{\beta}}}
\newcommand{\norm}[1]{\lVert#1\rVert}

\title{\bf On the convergence complexity of Gibbs samplers for a
  family of simple Bayesian random effects models}

\author{
  Bryant Davis \\ Department of Statistics \\
  University of Florida \\
  \texttt{davibf11@ufl.edu}
  \and
  James P. Hobert \\ Department of Statistics
  \\ University of Florida \\
  \texttt{jhobert@stat.ufl.edu}
}

\date{June 2020} 

\begin{document}
\maketitle

\begin{abstract}
The emergence of big data has led to so-called convergence complexity
analysis, which is the study of how Markov chain Monte Carlo (MCMC)
algorithms behave as the sample size, $n$, and/or the number of
parameters, $p$, in the underlying data set increase.  This type of
analysis is often quite challenging, in part because existing results
for fixed $n$ and $p$ are simply not sharp enough to yield good
asymptotic results.  One of the first convergence complexity results
for an MCMC algorithm on a continuous state space is due to
\citet{YR2019}, who established a mixing time result for a Gibbs
sampler (for a simple Bayesian random effects model) that was
introduced and studied by \citet{Rosenthal1996}.  The asymptotic
behavior of the spectral gap of this Gibbs sampler is, however, still
unknown.  We use a recently developed simulation technique
\citep{TraceClass} to provide substantial numerical evidence that the
gap is bounded away from 0 as $n \rightarrow \infty$.  We also
establish a pair of rigorous convergence complexity
results for two different Gibbs samplers associated with a
generalization of the random effects model considered by
\citet{Rosenthal1996}.  Our results show that, under strong regularity
conditions, the spectral gaps of these Gibbs samplers converge to 1 as
the sample size increases.
\end{abstract}

\section{Introduction}
\label{sec:intro}

Markov chain Monte Carlo (MCMC) is one of the most commonly used tools
in modern Bayesian statistics.  It is well known that the practical
performance of an MCMC algorithm is directly related to the speed at
which the underlying Markov chain converges to its stationary
distribution.  Over the last three decades, a great deal of work has 
been done to establish the convergence properties of many different 
practical Monte Carlo Markov chains.  Recently, with the emergence of 
big data, interest has shifted away from the analysis of individual 
Markov chains (for fixed data sets), and towards the study of how 
algorithms behave as the sample size, $n$, and/or the number of 
parameters, $p$, in the data set increase.  This type of study, which 
is called convergence complexity analysis, is often quite challenging, 
in part because the dimension of the Markov chain typically increases 
as $n$ and/or $p$ grow, and existing results for fixed $n$ and $p$ 
are simply not sharp enough to yield good asymptotic results 
\citep[see, e.g.,][]{RajSparks2015}.  Despite these difficulties, there 
has been a flurry of recent work on convergence complexity, which 
includes \citet{RajSparks2015}, \citet{Yang2016}, \citet{qin_probit}, 
\citet{YR2019}, \citet{Qin2019}, and \citet{je}.  Some of these papers 
analyze mixing times, while others focus on convergence rates.  In this 
paper, we study the asymptotic behavior of the spectral gaps of Gibbs 
samplers for a family of simple Bayesian random effects models.

One of the first convergence complexity results for MCMC (on a
continuous state space) was developed by \citet{YR2019}, who studied a
Gibbs sampler that was introduced and analyzed by
\citet{Rosenthal1996}.  Consider the following simple random effects
model:
\begin{equation}
  \label{eq:brem}
  y_i = \theta_i + e_i \,, \;\;\; i = 1, \ldots, n \;,
\end{equation}
where the components of $\bth =
(\theta_1,\dots,\theta_n)^T$ are iid $\mbox{N}(\mu,A)$, the components
of $\boldsymbol{e} = (e_1,\dots,e_n)^T$ are iid $\mbox{N}(0,V)$, and
$\bth$ is independent $\boldsymbol{e}$.  The error
variance, $V$, is assumed known.  We take $A$ and $\mu$ to be
\textit{a priori} independent with
\begin{equation}
  \label{eq:bp}
  \pi(\mu) \propto 1 \;\;\; \mbox{and} \;\;\; A \sim \mbox{IG}(a, b)
  \;,
\end{equation}
where $a,b>0$, and we say $X \sim \mbox{IG}(a, b)$ if its density is
proportional to $x^{-a-1} e^{-b/x} I_{(0,\infty)}(x)$.  Denote the
resulting posterior density as $\pi(\bth,\mu,A \mid \by)$, where $\by
= (y_1,\dots,y_n)^T$.  Consider a two-block Gibbs sampler with Markov
transition density (Mtd) given by
\begin{equation}
  \label{eq:mtd_r}
  k(\mu',A',\bth' \mid \mu,A,\bth) = \pi(\bth' \mid \mu', A', \by) \,
  \pi(\mu', A' \mid \bth, \by) \;,
\end{equation}
and let $\{\bth_m,(\mu_m,A_m)\}_{m=0}^\infty$ denote the corresponding
$(n+2)$-dimensional Markov chain.  (See Section~\ref{sec:Ros_Gibbs}
for the specific forms of the conditional densities.)
\citet{Rosenthal1996} used drift \& minorization (d\&m) conditions to
study the convergence properties of this chain in the case where $n$
is fixed.  Unfortunately, as is the case for many d\&m-based results,
\pcite{Rosenthal1996} bounds are not sharp enough to provide useful
information about the behavior of the chain as $n \rightarrow \infty$.
\citet{YR2019} developed a modified version of \pcite{Rosenthal1995}
general bound (on the total variation distance to stationarity), and
used it to establish a convergence complexity result concerning the
mixing time of the Gibbs sampler.  In particular, they proved that the
number of iterations required to get the total variation distance to
stationarity below a prespecified threshold is constant as $n
\rightarrow \infty$.  A precise statement of their result is given in
Section~\ref{sec:Ros_Gibbs}.

While \pcite{YR2019} mixing time result is certainly a step in the
right direction, it doesn't provide any concrete information about the
behavior of the spectral gap of the Gibbs sampler as $n \rightarrow
\infty$.  One of our main contributions in this paper is to provide
substantial numerical evidence that the spectral gap remains strictly
positive as $n \rightarrow \infty$.  Our analysis centers on the
marginal Markov chain, $\{\bth_m\}_{m=0}^\infty$, which is known to
converge at the same rate as the full Gibbs chain,
$\{\bth_m,(\mu_m,A_m)\}_{m=0}^\infty$
\citep{diac:khar:salo:2008,robe:rose:2001,Roman2014}.  Because 
$\{\bth_m\}_{m=0}^\infty$ is the marginal of a two-block Gibbs
chain, the corresponding Markov operator is self-adjoint and positive
\citep{liu:wong:kong:1994}.  We prove that this operator is also
\textit{trace-class}, which allows us to apply the simulation method
of \citet{TraceClass} to estimate its spectral gap.  We perform a
large scale numerical study on seven different simulated data sets,
each of size $n=10^7$, to gain an understanding of how the the
spectral gap behaves as $n \rightarrow \infty$.  Our results suggest 
that the gap is bounded away from zero as $n \rightarrow \infty$.

Our second contribution is a pair of rigorous convergence
complexity results for Gibbs samplers associated with a generalization
of \eqref{eq:brem}.  Indeed, consider a version of \eqref{eq:brem}
with replicates:
\begin{equation}
  \label{eq:bremr}
  y_{ij} = \theta_i + e_{ij} \,, \;\;\; i = 1, \ldots, n\, , \;\;\; j
  = 1, \ldots, r \;,
\end{equation}
where the components of $\bth$ are iid $\mbox{N}(\mu,A)$, the
components of $\boldsymbol{e} = (e_{11},\dots,e_{nr})^T$ are iid
$\mbox{N}(0,V)$, $\bth$ is independent $\boldsymbol{e}$, and $V$ is
known.  We consider two different priors.  The first is \eqref{eq:bp},
and in the second, the flat prior on $\mu$ is replaced with a normal
shrinkage prior whose variance decreases as $n,r \rightarrow \infty$.
A recently developed method for studying convergence rates via
Wasserstein distance \citep{Qin2019} is employed to analyze the
corresponding two-block Gibbs samplers.  We prove that, in each case, 
under a strong assumption about the rate at which $r = r(n)$ grows 
with $n$, not only is the spectral gap bounded away from 0 as $n \to 
\infty$, it actually converges to 1.

We made several serious attempts to use the Wasserstein-based
techniques mentioned above to study \pcite{Rosenthal1996} chain, but
we were unable to make any headway.  Thus, one surprise that comes out
of our study is that it's apparently easier to analyze the Gibbs
samplers associated with the \textit{more complex} likelihood
\eqref{eq:bremr} than it is to analyze \pcite{Rosenthal1996} chain.
So the Gibbs sampler that Yang \& Rosenthal chose to study (presumably
because of its simple form) turns out to be a relatively tough nut to
crack.

The remainder of the paper has the following organization.
Section~\ref{sec:back} provides the requisite background on Markov
chain convergence in both total variation and Wasserstein distances.
The Gibbs sampler studied by \citet{Rosenthal1996} and \citet{YR2019}
is the topic of Section~\ref{sec:Ros_Gibbs}.  In
Section~\ref{sec:rep}, we analyze the two Gibbs samplers associated
with the likelihood \eqref{eq:bremr}.  Section~\ref{sec:discussion}
contains some discussion.  All of the proofs are relegated to the
Appendix.

\section{Markov Chain Background}
\label{sec:back}

Suppose that $\X \subset \mathbb{R}^q$ and let $\mathcal{B}$ denote
its Borel $\sigma$-algebra.  Let $K:\X \times \mathcal{B} \to [0,1]$
be a Markov transition kernel (Mtk).  For any $m \in \mathbb{N} :=
\{1,2,3,\dots\}$, let $K^m$ be the $m$-step transition kernel, so that
$K^1 = K$.  For any probability measure $\mu: \mathcal{B} \to [0,1]$
and measurable function $f: \X \to \mathbb{R}$, denote $\int_{\X} f(x)
\, \mu(\df x)$ by $\mu f$, $\int_{\X} \mu(\df x) K^m(x, \cdot)$ by $\mu
K^m (\cdot)$, and $\int_{\X} K^m(\cdot, \df x) f(x)$ by $K^m
f(\cdot)$.  We often write $K_x^m(\cdot)$ instead of $K^m(x,\cdot)$.
Also, let $L^2(\mu)$ denote the set of measurable, real-valued
functions on $\X$ that are square integrable with respect to $\mu(\df
x)$.

Assume that the Markov chain corresponding to $K$ is Harris ergodic
(irreducible, aperiodic, and positive Harris recurrent), so it
converges to a unique stationary distribution, which we denote by
$\Pi$.  The goal of convergence analysis is to understand how fast
$\mu K^m$ converges to~$\Pi$ as~$m \to \infty$ for a large class of
$\mu$s.  The difference between $\mu K^m$ and $\Pi$ is usually
measured using the total variation distance, which is defined as
follows.  For two probability measures on $(\X, \mathcal{B})$, $\mu$
and $\nu$, their total variation distance is
\[
d_{\mbox{\scriptsize{TV}}} (\mu,\nu) = \sup_{A \in \mathcal{B}} \, [\mu(A) -
  \nu(A)] \,.
\]
The Markov chain defined by $K$ is \textit{geometrically ergodic} if
there exist $\rho < 1$ and $M: \X \to [0,\infty)$ such that, for each
  $x \in \X$ and $m \in \mathbb{N}$,
\begin{equation}
  \label{eq:ge}
  d_{\mbox{\scriptsize{TV}}}(K_x^m, \Pi) \leq M(x) \, \rho^m \,.
\end{equation}
Define the \textit{geometric convergence rate} of the chain as
\[
\rho_* = \inf \big\{ \rho \in [0,1] : \text{\eqref{eq:ge} holds for 
  some } M: \X \to [0,\infty) \big\} \;.
\]
Clearly, the chain is geometrically ergodic if and only if $\rho_* <
1$.

The space of functions $L^2(\Pi)$ is a Hilbert space with inner
product $\langle f,g \rangle = \int_\X f(x) \, g(x) \, \Pi(\df x)$ and norm of $f$ given by $\sqrt{\langle f,f \rangle}$.  The Mtk $K$ defines an
operator $K: L^2(\Pi) \rightarrow L^2(\Pi)$ that maps $f \in L^2(\Pi)$
to $Kf$.  If $K$ is reversible with respect to $\Pi$, then the Markov
operator $K$ is self-adjoint, and the Markov chain defined by $K$ is
geometrically ergodic if and only if the operator possesses a spectral
gap \citep{RobertsTweedie,RR1997}.  (For a nice overview of this
theory, see \citet{JerisonPHD}.)  If, in addition to being
self-adjoint, the Markov operator $K$ is also positive and compact,
then for every probability measure $\nu: \mathcal{B} \to [0,1]$ that
(is absolutely continuous with respect to $\Pi$ and) satisfies
$\int_{\X} (\df \nu/\df \Pi)^2 \, \df \Pi < \infty$, there exists a
constant $M_{\nu} < \infty$ such that
\begin{equation}
  \label{eq:L2ge}
  d_{\mbox{\scriptsize{TV}}}(\nu K^m, \Pi) \leq M_{\nu} \, \lambda_*^m
  \,,
\end{equation}
where $\lambda_*$ is the second largest eigenvalue of $K$.  (In this
context, the spectral gap is $1-\lambda_*$.)  It's also known that
$\lambda_* \le \rho_*$ \citep{RR1997}.

The standard method of developing upper bounds on $\rho_*$ requires
the construction of drift and minorization (d\&m) conditions for the
chain under study
\citep{Rosenthal1995,roberts2004general,baxendale2005renewal}.  It is
well known the d\&m-based methods are often overly conservative,
especially in high-dimensional situations \citep[see,
  e.g.,][]{RajSparks2015,Qin2020}.  There is mounting evidence
suggesting that convergence complexity analysis becomes more tractable
when total variation distance is replaced with an appropriate
Wasserstein distance \citep[see,
  e.g.,][]{hairer2011asymptotic,durmus2015quantitative,Qin2019}.  In
the remainder of this section, we describe a method of bounding
$\rho_*$ via Wasserstein distance.

Let $\phi(\cdot,\cdot)$ denote the usual Euclidean distance on
$\mathbb{R}^q$, i.e., $\phi(x, y) = \norm{x-y},$ and assume that 
$(\X,\phi)$ is a Polish metric space. For two probability measures on 
$(\X, \mathcal{B})$, $\mu$ and $\nu$, their Wasserstein distance is 
defined as
\[
d_{\mbox{\scriptsize{W}}}(\mu, \nu) = \underset{\xi \in \tau(\mu,
  \nu)}{\inf} \int_{\X \times \X} \norm{x-y} \, \xi(\df x, \df y) \;,
\]
where $\tau(\mu, \nu)$ is the set of all couplings of $\mu$ and $\nu$,
that is, the set of all probability measures $\xi(\cdot, \cdot)$ on
$(\X \times \X, \mathcal{B} \times \mathcal{B})$ having marginals
$\mu$ and $\nu$.  One way to bound the Wasserstein distance between
$K_x^m$ and $K_y^m$ is via coupling, and coupling is often achieved
through random mappings, which we now describe.  On a probability
space $(\Omega, \mathcal{F}, P)$, let $\theta: \Omega \to \Theta$ be a
random element, and let $\tilde{f}: \X \times \Theta \to \X$ be a
Borel measurable function.  Define $f(x) = \tilde{f}(x, \theta)$ for
all $x \in \X$.  Then $f$ is called a random mapping on $\X$.  The
evolution of a Markov chain can often be viewed as being governed by a
random mapping.  If $f(x) \sim K_x(\cdot)$ for all $x \in \X$, then we
say that $f$ induces $K$.  For example, suppose that $\X = \mathbb{R}$
and $K(x,dy) = (2 \pi)^{-1/2} \exp \{ - (y - x/2)^2/2 \} \, dy$.  Let
$Z$ be standard normal, and define $\tilde{f}(x,Z) = x/2 + Z$.  Then
the random mapping $f(x) = x/2 + Z$ induces $K$.

Assuming that $f$ induces $K$, let $\{ f_i \}_{i=1}^{\infty}$ be iid
copies of $f$, and let $F_m = f_m \circ f_{m-1} \circ \cdots \circ
f_1$ for $m \ge 1$. Then, for all $x, y \in \X$, $\big \{ (F_{m}(x),
F_{m}(y)) \big\}_{m=0}^{\infty}$ defines a Markov chain such that
$(F_{m}(x), F_{m}(y)) \in \tau(K_x^m, K_y^m)$ for all $m \ge 1$.  The
following result is well known \citep[see, e.g.,][]{Ollivier2009}.

\begin{proposition}
  \label{prop:ol}
  Assume that $c(x) = \int_\X \norm{x-y} \, K_x(dy) < \infty$ for all
  $x \in \X$.  Suppose that the random mapping $f$ induces $K$, and
  that there exists a $\gamma < 1$ such that, for every $x, y \in \X$,
\begin{equation*}
  \mbox{E} \, \norm{f(x)-f(y)} \le \gamma \, \norm{x-y} \;.
\end{equation*}
Then for each $x \in \X$ and each $m \in \mathbb{N}$, we have
\[
d_{\mbox{\scriptsize{W}}}(K^m_x,\Pi) \le \frac{c(x)}{1-\gamma} \,
\gamma^m \;.
\]
\end{proposition}

The next result provides a connection between Wasserstein distance and
total variation distance.

\begin{theorem}[\citet{Madras2010}]
  \label{thm:ms}
  Assume that $K_x(\cdot)$ has a density $k(\cdot | x)$ with respect
  to some dominating measure $\mu$ for all $x \in \X$.  If there
  exists a constant $C < \infty$ such that, for all $x, y \in \X$,
\[
\int_\X \big \lvert k(z \mid x) - k(z \mid y) \big \rvert \, \mu(\df z)
\le C \, \norm{x-y} \;,
\]
then, for all $m \in \{2,3,4,\dots\}$, we have
\[
d_{\mbox{\scriptsize{TV}}}(K_x^m, \Pi) \le \frac{C}{2} \,
d_{\mbox{\scriptsize{W}}}(K_x^{m-1}, \Pi) \;.
\]
\end{theorem}

Suppose that we are able to show that
$d_{\mbox{\scriptsize{W}}}(K^m_x,\Pi) \le M(x) \, \gamma^m$ where
$\gamma \in [0,1)$ and $M: \X \rightarrow [0,\infty)$.  Then, if the
    conditions in Theorem~\ref{thm:ms} are satisfied, we have $\rho_*
    \le \gamma$.  Finally, the following result provides a tractable
    upper bound for $\mbox{E} \, \norm{f(x)-f(y)}$.

\begin{lemma}[\citet{Qin2019}]
  \label{lem:qin}
  Assume that $\X \subset \mathbb {R}^q$ is convex and that $f$ is a 
  random mapping on $\X$.  Let $x, y \in \X$ be fixed.  Suppose that, 
  with probability 1, $\frac{\df}{\df t} f(x + t(y-x))$, as a function 
  of $t \in [0,1]$, exists and is integrable. Then
\[
\mbox{E} \, \norm{f(x) - f(y)} \le \underset{t \in [0,1]}{\sup}
\mbox{E} \, \Big \lVert \frac{\df}{\df t} f(x + t(y-x)) \Big \rVert
\;.
\]
\end{lemma}

The reader may wonder why we are concerned with converting Wasserstein
bounds into total variation bounds, instead of simply being satisfied
with convergence in Wasserstein distance.  One reason is the existence
of central limit theorems (CLTs), which are extremely important for
the application of MCMC in practice \citep[see,
  e.g.,][]{flegal2011implementing}.  Let $\{X_m\}_{m=0}^\infty$ denote
the Markov chain corresponding to $K$, and suppose that $f:\X
\rightarrow \mathbb{R}$ is such that $\Pi \, |f| < \infty$.  Then
because the chain is Harris ergodic, $\hat{f}_m := m^{-1}
\sum_{i=0}^{m-1} f(X_i)$ is a strongly consistent estimator of $\Pi
f$.  If, in addition, $K$ satisfies \eqref{eq:ge} (so the chain is
geometrically ergodic with respect to total variation distance), and
$\Pi \, |f|^{2 + \delta} < \infty$ for some $\delta>0$, then $\sqrt{n}
(\hat{f}_m - \Pi f)$ has a Gaussian limit distribution.  On the other
hand, if we replace total variation convergence with Wasserstein
convergence, then the $2+\delta$ moment is no longer sufficient for a
CLT, and stronger conditions on $f$ (such as $f$ being a Lipschitz
function) are required \citep{wassersteinCLT}.

\section{Rosenthal's Gibbs Sampler}
\label{sec:Ros_Gibbs}

\subsection{What is known?}
\label{ssec:known}

As in the Introduction, let $\pi(\bth,\mu,A \mid \by)$ denote the
posterior density that results when the likelihood associated with
\eqref{eq:brem} is combined with the prior \eqref{eq:bp}.  We now
describe the two conditionals that define \eqref{eq:mtd_r}.  Of
course, $\pi(A \mid \bth, \by) \propto \pi(\bth, A \mid \by) =
\int_\mathbb{R} \pi(\bth,\mu,A \mid \by) \, \df \mu$, and it follows
that
\[
A \mid \bth, \by \sim \mbox{IG} \bigg( a + \frac{n-1}{2},
b + \frac{1}{2} \sum_{i = 1}^{n} (\theta_i - \bar{\theta})^2 \bigg)
\;,
\]
where $\bar{\theta}$ is the mean of the $\theta_i$s.  Also,
\[
\mu \mid \bth, A, \by \sim \mbox{N} (\bar{\theta}, A/n ) \;.
\]
Clearly, the product of these two conditional densities equals
$\pi(\mu,A \mid \bth,\by)$.  Thus, given $\bth$, we can sample from
$\pi(\mu,A \mid \bth,\by)$ by first drawing from $\pi(A \mid
\bth,\by)$, and then drawing from $\pi(\mu \mid \bth, A,\by)$.  It's
also easy to show that, conditional on $A$, $\mu$, and $\by$, the
elements of $\bth$ are independent with
\[
\theta_i \mid \bth_{-i}, A, \mu, \by \sim \mbox{N} \bigg( \frac{V\mu +
  A y_i}{A+V}, \frac{AV}{A+V} \bigg) \;.
\]
The fact that the Mtd \eqref{eq:mtd_r} is strictly positive on $\left( 
\mathbb{R} \times \mathbb{R}_+ \times \mathbb{R}^n \right) \times 
\left( \mathbb{R} \times \mathbb{R}_+ \times \mathbb{R}^n \right)$ 
implies that the Markov chain it defines is Harris ergodic (see, e.g., 
\citet{asmusglynn2011}), where $\mathbb{R}_+ := (0, \infty)$.

Note that we are actually considering an entire family of chains here.
Indeed, on the left-hand side of \eqref{eq:mtd_r} we are suppressing
dependence on the sample size $n \in \mathbb{N}$, the data $\by \in
\mathbb{R}^n$, the known error variance $V$, and the hyperparameter
$(a,b) \in \mathbb{R}_+ \times \mathbb{R}_+$.  Since the convergence 
behavior of the Gibbs chain may depend on $(n,\by,V,a,b)$, we will 
denote the geometric convergence rate by $\rho_*(n,\by,V,a,b)$. 
\citet{Rosenthal1996} used drift and minorization (d\&m) conditions 
in conjunction with results in \citet{Rosenthal1995} to establish 
that every member of the family is geometrically ergodic, and his 
results lead to an explicit function, $\hat{\rho}(n,\by,V,a,b)$, such 
that, for each fixed $(n,\by,V,a,b)$, $\rho_*(n,\by,V,a,b) \le 
\hat{\rho}(n,\by,V,a,b) < 1$.

Our interest centers on the convergence behavior of the Gibbs sampler
as $n \rightarrow \infty$.  \pcite{Rosenthal1996} upper bound,
$\hat{\rho}(n,\by,V,a,b)$, converges (rapidly) to 1 as $n \rightarrow
\infty$ \citep{YR2019}, which suggests that the chain may behave
poorly when $n$ is large.  However, as we shall see below, there is
actually a great deal of evidence pointing in the opposite direction.
One might be tempted to attribute this disconnect to the fact that
d\&m-based methods often break down in high dimensional situations,
but, as we now explain, increasing dimension is not the culprit.
Recall that the marginal chains $\{\mu_m,A_m\}_{m=0}^\infty$ and
$\{\bth_m\}_{m=0}^\infty$ have the same convergence rate as the full
Gibbs chain, $\{\bth_m,(\mu_m,A_m)\}_{m=0}^\infty$, and note that
$\{\mu_m,A_m\}_{m=0}^\infty$ always has dimension 2, regardless of
$n$.  The Mtd of this chain is given by
\begin{equation}
  \label{eq:k2}
  k_1(\mu', A' \mid \mu, A) = \int_{\mathbb{R}^n} \pi(\mu',A' \mid
  \bth, \by) \, \pi(\bth \mid \mu, A, \by) \, \, \df \bth \;.
\end{equation}
Because the integrand of $k_1$ contains the $n$-dimensional density
$\pi(\bth \mid \mu, A, \by)$, it's possible that the chain
$\{\mu_m,A_m\}_{m=0}^\infty$ is not completely immune to increasing
dimension.  Note, however, that $\pi(\mu,A \mid \bth,\by)$ depends on
$\bth$ only through two univariate functions of $\bth$: $\bar{\theta}$
and $\sum_{i=1}^n (\theta_i - \bar{\theta})^2$.  Thus, we can perform
a change of variables on the right-hand side of \eqref{eq:k2} that
reduces the dimension of the integral from $n$ to 2.  Of course, the
new integrand would involve the \textit{variable} $n$, but $n$ would
no longer represent a \textit{dimension}.  Thus, there is no sense in
which the chain $\{\mu_m,A_m\}_{m=0}^\infty$ depends on dimension $n$ 
in any way other than as a parameter. We have made multiple attempts 
to prove that $\rho_*(n,\by,V,a,b)$ is bounded away from 1 as $n 
\rightarrow \infty$ by analyzing each of the two marginal chains using 
both d\&m methods and Wasserstein methods, and all have been 
unsuccessful.  It seems quite difficult to get a handle on the 
asymptotic behavior of $\rho_*(n,\by,V,a,b)$, and the difficulty goes 
beyond increasing dimension.

\citet{YR2019} attacked this convergence complexity problem in a
different way.  Instead of focusing on the geometric convergence rate,
they studied the mixing time.  In particular, these authors showed that, 
under a weak assumption on the asymptotic behavior of $(n-1)^{-1} 
\sum_{i=1}^{n} \left( y_i - \bar{y} \right)^2$, and for a particular 
starting value $(\boldsymbol{\theta}_0, A_0, \mu_0)$, there exist an 
$N \in \mathbb{N}$, positive constants $C_1, C_2, C_3$ and $\gamma \in 
(0, 1)$ such that, for all $n \ge N$ and all $m$,
\begin{equation}
  \label{eq:yrb}
  d_{\mbox{\scriptsize{TV}}}(K_{(\boldsymbol{\theta}_0,A_0,\mu_0), n}^m, 
  \Pi_n) \leq C_1 \gamma^m + C_2 \frac{m(m+1)}{n} + C_3 
  \frac{m}{\sqrt{n}} \;,
\end{equation}
where $K_{(\boldsymbol{\theta}_0,A_0,\mu_0), n}^m$ denotes the $m$-step 
Mtk of the Gibbs sampler started at $(\boldsymbol{\theta}_0, A_0, \mu_0)$ 
based on sample size $n$, and $\Pi_n$ denotes the corresponding posterior 
distribution.  Note that the right-hand side of
\eqref{eq:yrb} is a decreasing function of $n$.  Thus, if for some
fixed $m'$ and $n' \ge N$, the total variation distance is less than
some threshold, then this remains so for all $n > n'$ with the same
$m'$.  So, in this sense, the mixing time is constant in $n$. While 
this result certainly suggests that the chain is
reasonably well-behaved when $n$ is large, it does not provide us with
any information about the asymptotic behavior of the geometric
convergence rate as $n \rightarrow \infty$.  Indeed, \eqref{eq:yrb} does not even
imply that the chain is geometrically ergodic. In the next subsection,
we apply a simulation technique developed in \citet{TraceClass} to
produce evidence suggesting that the spectral gap is bounded away from
0 as $n \rightarrow \infty$.

\subsection{A numerical investigation of the asymptotic properties of $\lambda_*$}
\label{ssec:numerical}

Recall the Markov operator, $K$, from Section~\ref{sec:back}.  If $K$
is self-adjoint, positive and compact, then it has a pure eigenvalue
spectrum, and the eigenvalues are all in the set $[0,1]$.  If, in
addition, the eigenvalues are summable, then $K$ is called
\textit{trace-class}.  (See \citet{TraceClass} for more details.)
\citet{TraceClass} provide a method of estimating the second largest
eigenvalue of such a $K$, which, as we know from \eqref{eq:L2ge},
dictates the rate of convergence.  Here's the basic idea.  Let
$\{\lambda_i\}_{i=0}^{\kappa}$ denote the non-zero eigenvalues of $K$,
in decreasing order, so $\lambda_0=1$, $\lambda_i \in (0,1)$ for all
$i \in \{1,2,\dots,\kappa\}$, and $\kappa$ could be $\infty$.  (In the
sequel, we use $\lambda_1$ and $\lambda_*$, interchangeably.)  Now,
fix a positive integer $l$, and define
\[
s_l = \sum_{i=0}^\kappa \lambda_i^l \;.
\]
The fact that the chain is trace-class implies that this sum is finite
for any $l \in \mathbb{N}$.  \citet{TraceClass} show that $u_l = (s_l
- 1)^{1/l}$ is an upper bound on $\lambda_*$, which decreases to
$\lambda_*$ as $l \rightarrow \infty$.  These authors also develop a
classical Monte Carlo estimator for $s_l$ that is asymptotically
normal, and this leads to an asymptotically normal estimator for
$u_l$.  We will apply this method to the marginal chain
$\{\bth_m\}_{m=0}^\infty$ whose Mtd is given by
\[
k_2(\bth' \mid \bth) = \int_{\mathbb{R}_{+}} \int_{\mathbb{R}}
\pi(\bth' \mid \mu, A, \by) \, \pi(\mu,A \mid \bth, \by) \, \df \mu \,
\df A \;.
\]
Again, the Markov operators associated with the marginal chains of any
two-block Gibbs sampler are reversible and positive, so all we have
left to do is to show that the Markov operator associated with $k_2$
is trace-class (which implies compactness).  By \pcite{TraceClass}
Theorem 2, it suffices to show that
\[
\int_{\mathbb{R}^{n}} k_2(\bth \mid
\bth) \, \df \bth < \infty \;.
\]
A proof of the following result is provided in Appendix~\ref{app:A}.
\begin{proposition}
  \label{pro:tc}
  The Markov operator defined by $k_2$ is trace-class whenever $n \ge
  3$.
\end{proposition}

In order to apply the Monte Carlo algorithm, we must specify an
\textit{auxiliary} density $\omega(\mu,A)$ that is positive (almost)
everywhere on $\mathbb{R} \times \mathbb{R}_+$.  We will use
$\omega(\mu,A) = \omega(\mu \mid A) \, \omega(A)$, where
\[
\omega(A) = \mbox{IG}(a,b) \hspace*{5mm} \mbox{and} \hspace*{5mm}
\omega(\mu \mid A ) = \mbox{N} \bigg( \bar{y}, \frac{(A+V)(A+4V)}{nA}
\bigg) \;.
\]
The (strongly consistent) Monte Carlo estimator of $s_l$ is given by
\begin{equation}
  \label{eq:mce}
  \hat{s}_l = \frac{1}{N} \sum_{i=1}^N \frac{\pi(\mu_i^*,A_i^* \mid
    \bth_i^*)}{\omega(\mu_i^*,A_i^*)} \;,
\end{equation}
where the random vectors $\{(\bth_i^*,\mu_i^*,A_i^*)\}_{i=1}^N$ are
iid and each is generated according to Algorithm 1.

\begin{algorithm}
    \vspace*{2.5mm}
    \begin{itemize}
        \item[1.] Draw $(\mu^{*}, A^{*}) \sim \omega(\cdot,\cdot)$. \\
        \item[2.] Given $(\mu^{*}, A^{*}) = (\mu, A)$, for $i = 1,
          \ldots, n$, draw
\[
\theta'_i \overset{ind}{\sim} \mbox{N} \bigg( \frac{V\mu + A y_i}{A +
  V}, \frac{AV}{A+V} \bigg) \;,
\]
and set $\bth' = (\theta'_1,\dots,\theta'_n)^T$. \\
        \item[3.] If $l = 1$, set $\bth^* = \bth'$. If $l \ge 2$, draw
          $\bth^* \sim k_2^{(l-1)}(\cdot \mid \bth')$ by running $l-1$
          iterations of the two-block Gibbs sampler.
    \end{itemize}
    \caption{Drawing $(\bth^*,\mu^*,A^*)$ in order to
      estimate $s_l$}
\end{algorithm}

By \pcite{TraceClass} Theorem 4, the Monte Carlo estimator
\eqref{eq:mce} has finite variance if the following condition is
satisfied:
\[
\int_{\mathbb{R}_+} \int_{\mathbb{R}} \int_{\mathbb{R}^n} \frac{\pi^3
  (A, \mu \mid \bth) \pi (\bth \mid A, \mu) }{\omega^2 (A, \mu)}
\, \df \bth \, \df \mu \, \df A < \infty \;.
\]
A proof of the following result is provided in Appendix~\ref{app:B}.
\begin{proposition}
  \label{pro:mcv}
  The Monte Carlo estimator \eqref{eq:mce} has finite variance
  whenever $n \ge 3$.
\end{proposition}

In order to estimate $u_l$, the upper bound on the second largest
eigenvalue of the Markov operator defined by $k_2$, we have to apply
Algorithm 1 $N$ times, where $N$ is the Monte Carlo sample size, and
within each iteration of Algorithm 1, we must generate $n$ univariate
normals $l$ times.  This becomes quite burdensome when $n$ is large,
which, unfortunately, is precisely the case on which we are focused.
However, there is a simple way of circumventing this problem.  Recall
that $\pi(\mu,A \mid \bth,\by)$ depends on $\bth$ only through two
univariate functions of $\bth$: $\bar{\theta}$ and $\sum_{i=1}^n
(\theta_i - \bar{\theta})^2$.  Moreover, if the components of $\bth$
are independent with
\[
\theta_i \sim \mbox{N} \bigg( \frac{V\mu + A y_i}{A+V}, \frac{AV}{A+V}
\bigg) \;,
\]
then it follows that $\bar{\theta}$ and $\sum_{i=1}^n (\theta_i -
\bar{\theta})$ are independent,
\[
\bar{\theta} \sim \mbox{N} \bigg( \frac{V\mu + A \bar{y}}{A+V},
\frac{AV}{n(A+V)} \bigg) \hspace*{4mm} \mbox{and} \hspace*{4mm}
\frac{(A+V)}{AV} \sum_{i=1}^n (\theta_i - \bar{\theta})^2 \sim
\chi^2_{n-1}(\phi) \;,
\]
where $\phi = (A \Delta)/(2V(A+V))$ is the non-centrality parameter,
and $\Delta = \sum_{i=1}^n (y_i - \bar{y})^2$.  Therefore, when
running Algorithm 1, each time we are required to make a draw from
$\pi(\bth \mid \mu, A, \by)$, which nominally requires making $n$
independent univariate normal draws, we can instead simply draw one
univariate normal and one non-central $\chi^2$.  This maneuver is a
massive time saver when $n$ and/or $l$ are large.

We now employ \pcite{TraceClass} method to gain some insight into the
behavior of the convergence rate of the $\bth$-chain as $n$ becomes
large.  Again, we have a family of chains indexed by $(n,\by,V,a,b)$,
so $\lambda_* = \lambda_*(n,\by,V,a,b)$.  Our idea is to consider a 
sequence of Gibbs samplers based on a growing data set (with $a$, $b$ 
and $V$ fixed) to study whether there is a noticeable relationship 
between the value of $\lambda_*$ and increasing dimension.  We 
simulated seven different data sets, that is, seven different versions 
of $\by$, each of length $10^7$.  The simulations were based on 
different values of $A$ and $V$.  In three of the the cases, we set 
$A = V$ with values $\{ 1, 10, 100 \}$, in two cases we took $A$ 
larger than $V$ ($A = 10$, $V = 1$ and $A = 100$, $V = 10$), and in 
the final two cases, we took $V$ larger than $A$ ($A = 1$, $V = 10$ 
and $A = 10$, $V = 100$).  For each of the seven configurations, we 
simulated $\theta_i \overset{iid}{\sim} \mbox{N}(0, A)$, $i=1,\dots,
10^7$, and then we simulated $y_i \overset{ind}{\sim} \mbox{N}
(\theta_i, V)$, $i = 1,\dots,10^7$. Then, for each of the seven 
configurations, we considered six different $\bth$-chains 
corresponding to six different samples sizes: $n = 10^2, 10^3, 
\ldots, 10^7$. We then applied \pcite{TraceClass} method to each of the 
six chains.  So, overall, we estimated an upper bound on $\lambda_*$ 
for 42 different Markov chains. Of course, in order to use 
\pcite{TraceClass} algorithm, we need to specify values for $a$ and 
$b$.  We simply chose $a$ and $b$ such that $b/(a-1)$ equals the value 
of $A$ that was used to simulate the data. (See Figure 1 for the exact 
values.)

Of course, there is still the issue of choosing the tuning parameter,
$l$.  \citet{TraceClass} recommend increasing $l$ until $u_l$ is
strictly less than 1.  In practice, one can observe $u_l$ steadily
fall as $l$ increases before hitting a point of volatility, where it
begins to produce unreliable estimates.  (The variance of the
estimator of $s_l$ is stable as $l \rightarrow \infty$, but that of
$u_l$ is not.)  We utilized a ``Goldilocks'' strategy, choosing values
of $l$ that were large enough to have $u_l$ appear to be a good
estimator of the upper bound but small enough to ensure the variance
of $u_l$ remains as low as possible.  In each case, we used a Monte
Carlo sample size of $N = 5,000,000$, that is, for each of the 42
different Markov chains that we studied, once we identified a
reasonable value of $l$, we used Algorithm 1 to produce $N =
5,000,000$ draws of $(\bth^*,\mu^*,A^*)$, and those were then used to
estimate $s_l$ (and $u_l$).

The results are presented in Figure 1.  There is one plot for each of
the seven configurations, and in each case, it appears that, as the
sample size, $n$, becomes large, the estimated upper bounds on
$\lambda_*$ approach an asymptote that is strictly below 1.  Note that
the values of $n$ in each plot increase by a factor of 10 each time.
It is clear that different underlying values of $a$, $b$, and $V$ can
result in different convergence rates, and that $\lambda_*$ can grow
as $n$ increases, but, in each case, $\lambda_*$ appears to be bounded
away from 1 as $n$ grows. Because each of the 42 estimates is based on 
a very large Monte Carlo sample size ($5 \times 10^6$), the standard 
errors are all relatively small, and certainly not large enough to 
change the takeaway that $\lambda_*$ seems to be bounded
away from 1.

\begin{figure}[h!]
  \begin{center}
  \includegraphics[width=\textwidth]{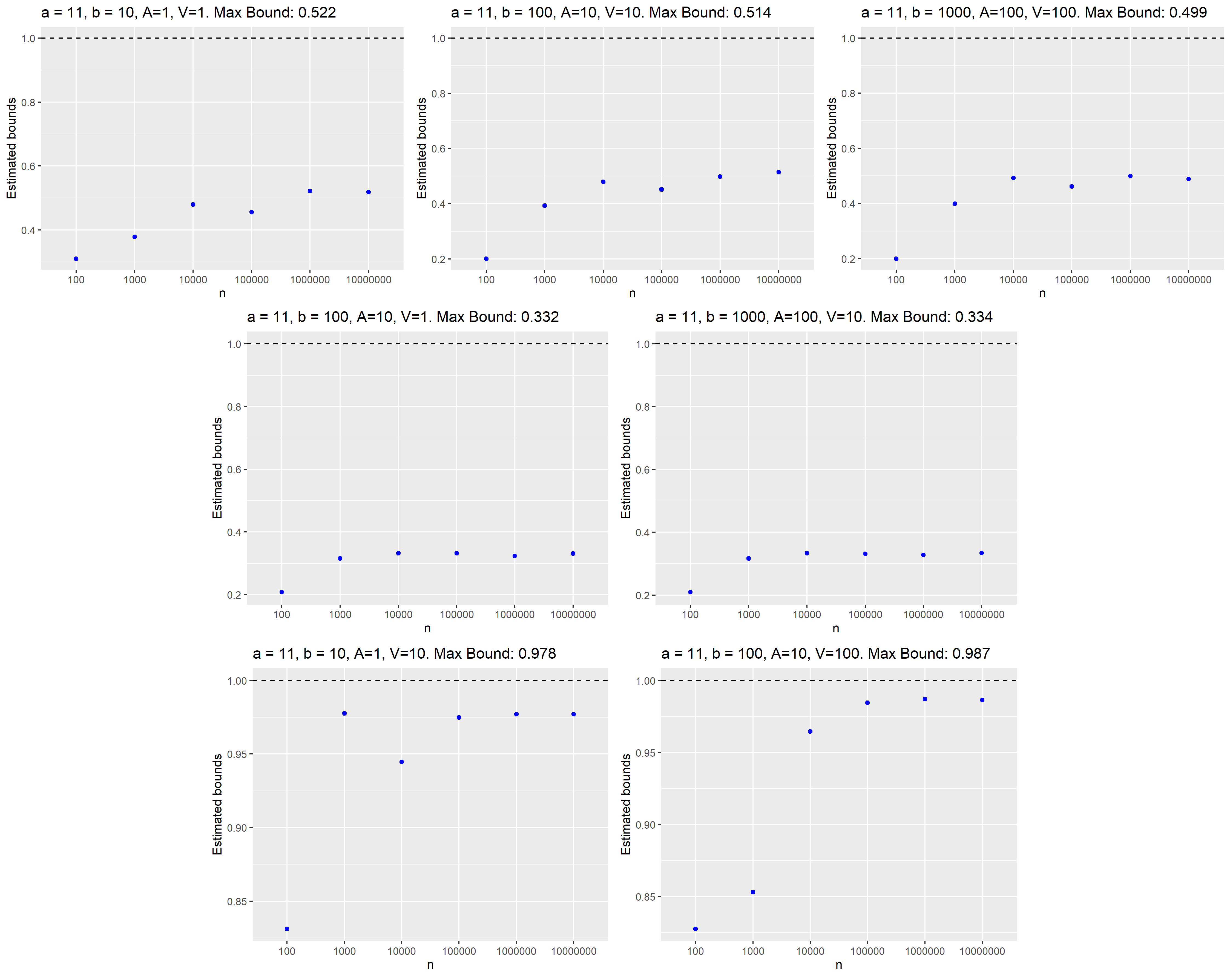}
  \caption{Plots of the Monte Carlo estimator of the upper bound $u_l$ 
    of $\lambda_*$ for each of 42 different Gibbs samplers.  There is 
    one plot for each of the seven simulated data sets.  The values of 
    $A$ and $V$ that were used to simulate the data are provided, as 
    are the values of $a$ and $b$ that were used to run the Gibbs 
    samplers.}
  \end{center}
\end{figure}

Our numerical work suggests that $\lambda_*$ is bounded away from 1 as 
$n \rightarrow \infty$, and if this is true, one would think that 
$\rho_*$ probably behaves similarly.  On the other hand, it is true that 
$\lambda_* \le \rho_*$, so, while it seems unlikely, it is possible that 
$\rho_*$ behaves poorly even when $\lambda_*$ does not. In the next 
section, we show that a more complex random effects model (containing 
\textit{replicates}) leads to Gibbs samplers that are actually easier 
to analyze than those studied in this section.

\section{Gibbs Samplers for Models with Replicates}
\label{sec:rep}

\subsection{An alternative blocking strategy}
\label{ssec:block}

Here we consider the posterior that results when we combine the
likelihood defined by \eqref{eq:bremr} with the prior \eqref{eq:bp}.
It turns out to be more convenient to work with a simple
transformation of the resulting posterior.  Let $\eta_0 = \sqrt{n}
\mu$, $\eta_i = \theta_i - \mu$, $i=1,\dots,n$, and $B = 1/A$.  Then
the new posterior density is given by
\[
\pi(\boldsymbol{\eta},B \mid \by) \propto B^{a + \frac{n}{2} - 1} \exp
\Bigg \{ -\frac{U}{2} \sum_{i=1}^{n} \sum_{j=1}^{r} \bigg( y_{ij} -
\frac{\eta_0}{\sqrt{n}} - \eta_i \bigg)^2 - B \bigg( b + \frac{1}{2}
\sum_{i=1}^{n} \eta_i^2 \bigg) \Bigg\} \, I_{\mathbb{R}_+}(B) \;,
\]
where $\bet = (\eta_0,\dots,\eta_n)^T$ and $U = 1/V$.  Routine
calculations show that
\[
\eta_0 \mid B, \by \sim \mbox{N} \bigg( \sqrt{n} \, \bar{y}, \frac{B +
  rU}{r BU} \bigg) \;,
\]
where $\bar{y} = (nr)^{-1} \sum_{i,j} y_{ij}$.  Moreover, conditional
on $(\eta_0,B,\by)$, $\eta_1,\dots,\eta_n$ are independent with
\[
\eta_i \mid \bet_{-i}, B, \by \sim \mbox{N} \bigg( \frac{rU}{B + r U}
\Big( \bar{y}_i - \frac{\eta_0}{\sqrt{n}} \Big), \frac{1}{B + r U}
\bigg) \;.
\]
Thus, we can draw from $\pi(\boldsymbol{\eta} \mid B , \by)$ in a
sequential manner.  Finally, it's easy to show that
\[
B \mid \boldsymbol{\eta} , \by \sim \mbox{Gamma} \bigg( a + \frac{n}{2},
b + \frac{1}{2} \sum_{i=1}^{n} \eta_i^2 \bigg) \;.
\]
Consider the two-block Gibbs sampler with blocks $\boldsymbol{\bet}$
and $B$.  Note that here, unlike in Section~\ref{sec:Ros_Gibbs}, all
of the location parameters are in one block, and the second block has
just a single, univariate parameter.  We will study the
$\bet$-marginal of this two-block sampler, which we denote by
$\{\bet_m\}_{m=0}^\infty$.  Let $K$ denote its Mtk and $\Pi$ its
stationary distribution.  The corresponding Mtd is given by
\[
k(\bet' \mid \bet) = \int_0^\infty \pi(\bet' \mid B, \by) \, \pi(B
\mid \bet, \by) \, \df B \;.
\]

We now describe a random mapping that induces
$\{\bet_m\}_{m=0}^\infty$.  Let $\bar{y}_i = \frac{1}{r} \sum_{j =
  1}^{r} y_{ij}$.  Also, let $J$ and $\{N_i\}_{i=0}^n$ be independent
and such that $J \sim \mbox{Gamma} \big( a + \frac{n}{2}, 1 \big)$ and
$\{N_i\}_{i=0}^n$ are iid $\mbox{N}(0, 1)$.  Fix $\bet$ and define
\begin{align*}
\tilde{B}^{(\bet)} & = \frac{J}{b + \frac{1}{2} \sum_{i=1}^n \eta_i^2}
\\ \tilde{\eta}_0^{(\bet)} & = \sqrt{n} \, \bar{y} +
\sqrt{\frac{\tilde{B}^{(\bet)} + r U}{r \tilde{B}^{(\bet)}U}}N_0
\\ \tilde{\eta}_{i}^{(\bet)} &= \frac{r U}{\tilde{B}^{(\bet)} + r U}
\left( \bar{y}_i - \frac{\tilde{\eta}_{0}^{(\bet)}}{\sqrt{n}} \right)
+ \sqrt{\frac{1}{\tilde{B}^{(\bet)} + r U }}N_i \,, \;\;\;
i=1,2,\dots,n \;.
\end{align*}
Now let $f(\bet) = \big( \tilde{\eta}_0^{(\bet)},
\tilde{\eta}_1^{(\bet)}, \dots, \tilde{\eta}_n^{(\bet)} \big)^T$.
It's clear that $f(\bet) \sim K_\bet(\cdot)$.  Our main result
involves the following conditions:
\begin{itemize}
\item[(A1)] $\frac{r(n)^2}{n^3} \to \infty$ as $n \rightarrow \infty$.
\item[(A2)] For all large $n$, $n^{-1} \sum_{i=1}^{n} (\bar{y}_i -
  \bar{y})^2 < C$, for some $C < \infty$.
\end{itemize}

A proof of the following result is given in Appendix~\ref{app:C}.

\begin{proposition}
  \label{prop:block}
  Let $\rho_*(n,r(n),\by,U,a,b)$ denote the geometric convergence rate of
  $\{\bet_m\}_{m=0}^\infty$, and assume that (A1) and (A2) hold. Then
  $\rho_{*}(n,r(n),\by,U,a,b) \rightarrow 0$ as $n \rightarrow \infty$.
\end{proposition}

\begin{remark}
  \citet{Qin2019} proved a similar result for a more complex model in
  which $V = 1/U$ is considered unknown, and has an $\mbox{IG}$ prior.
  Our proof is quite similar to theirs.
\end{remark}

Proposition~\ref{prop:block} constitutes a strong convergence
complexity result.  Indeed, not only is the geometric convergence rate
bounded below 1 as $n,r(n) \rightarrow \infty$, but it actually converges
to 0.  Of course, $\frac{r(n)^2}{n^3} \to \infty$ is a strong assumption.

\subsection{A shrinkage prior}
\label{ssec:shrink}

Here we consider the same model as in Subsection~\ref{ssec:block},
except that we change the prior on $\mu$.  Instead of the flat prior
on $\mu$, we employ a shrinkage prior: $\mu \sim \mbox{N}(w,z^{-1})$,
where $w$ is fixed, but $z = z(n)$.  It turns out to be more
convenient to work with a simple transformation.  Let $\beta_i =
\theta_i - \mu$, $i=1,\dots,n$, and $B = 1/A$.  Then the resulting
posterior density is given by
\[
\pi(\bbe,\mu,B \mid \by) \propto B^{a + \frac{n}{2} - 1} \exp \Bigg \{
-\frac{U}{2} \sum_{i=1}^{n} \sum_{j=1}^{r} \big( y_{ij} - (\beta_i +
\mu) \big)^2 - B \bigg( b + \frac{1}{2} \sum_{i=1}^{n} \beta_i^2
\Bigg) - \frac{z}{2} (\mu-w)^2 \Bigg\} \, I_{\mathbb{R}_+}(B) \;,
\]
where $\bbe = (\beta_1,\dots,\beta_n)^T$ and $U = 1/V$.  It's easy to
see that, conditional on $(\bbe,\by)$, $\mu$ and $B$ are independent,
and that
\[
B \mid \mu, \bbe, \by \sim \mbox{Gamma} \bigg( a + \frac{n}{2}, b +
\frac{1}{2} \sum_{i=1}^{n} \beta_i^2 \bigg) \;,
\]
and
\[
\mu \mid B, \bbe, \by \sim \mbox{N} \bigg( \frac{nr U ( \bar{y} -
  \bar{\beta}) + z w}{nrU + z}, \frac{1}{nrU + z} \bigg) \;.
\]
Conditional on $(\mu,B,\by)$, the components of $\bbe$ are independent
with
\[
\beta_i \mid \bbe_{-i}, \mu, B, \by \sim \mbox{N} \bigg( \frac{rU}{B +
  rU} (\bar{y}_i - \mu), \frac{1}{B + rU} \bigg) \;.
\]
Consider the two-block Gibbs sampler with blocks $(\mu,B)$ and $\bbe$.
We will study the $\bbe$-marginal of this two-block sampler, which we
denote by $\{\bbe_m\}_{m=0}^\infty$.  Let $K$ denote its Mtk and $\Pi$
its stationary distribution.  The corresponding Mtd is given by
\[
k(\bbe' \mid \bbe) = \int_0^\infty \int_{\mathbb{R}} \pi(\bbe' \mid
\mu, B, \by) \, \pi(\mu, B \mid \bbe, \by) \, \df \mu \, \df B \;.
\]

We now describe a random mapping that induces
$\{\bbe_m\}_{m=0}^\infty$.  Let $J$ and $\{N_i\}_{i=0}^n$ be
independent and such that $J \sim \mbox{Gamma} \big( a + \frac{n}{2},
1 \big)$ and $\{N_i\}_{i=0}^n$ are iid $\mbox{N}(0, 1)$.  Fix $\bbe$
and define
\begin{align*}
\tilde{B}^{(\bbe)} & = \frac{J}{b + \frac{1}{2} \sum_{i=1}^n
  \beta_i^2} \\ \tilde{\mu}^{(\bbe)} & = \frac{nr U (\bar{y} -
  \bar{\beta}) + z w}{nrU + z} + \frac{N_0}{\sqrt{nrU + z}}
\\ \tilde{\beta}_{i}^{(\bbe)} &= \frac{r U}{\tilde{B}^{(\bbe)} + r U}
\big( \bar{y}_i - \tilde{\mu}^{(\bbe)} \big) +
\frac{N_i}{\sqrt{\tilde{B}^{(\bbe)} + r U }} \,, \;\;\; i=1,2,\dots,n
\;.
\end{align*}
Now let $f(\bbe) = \big( \tilde{\beta}_1^{(\bbe)},
\tilde{\beta}_2^{(\bbe)}, \dots, \tilde{\beta}_n^{(\bbe)} \big)^T$.
It's clear that $f(\bbe) \sim K_\bbe(\cdot)$.  Our main result
involves the following conditions:
\begin{itemize}
\item[(A3)] $\frac{z(n)}{n \cdot r(n)} \to \infty$ as $n \rightarrow \infty$.
\item[(A4)] $|\bar{y}|$ is bounded above for large $n$.
\end{itemize}
A proof of the following result is given in Appendix~\ref{app:D}.

\begin{proposition}
  \label{prop:shrink}
  Let $\rho_*(n,r(n),\by,U,a,b,w,z(n))$ denote the geometric convergence
  rate of $\{\bbe_m\}_{m=0}^\infty$, and assume that (A1)-(A4) hold.
  Then $\rho_{*}(n,r(n),\by,U,a,b,w,z(n)) \rightarrow 0$ as $n
  \rightarrow \infty$.
\end{proposition}

\section{Discussion}
\label{sec:discussion}

It should be noted that all of the Gibbs samplers analyzed in this
paper can be considered ``toys'' in the following sense.  In each
case, it is possible to make an exact draw from the posterior by
drawing one $(n+1)$-dimensional multivariate normal random vector, and
one random variable from an intractable \textit{univariate} density.
For example, consider the posterior density, $\pi(\bth,\mu,A \mid
\by)$, from Section~\ref{sec:Ros_Gibbs}.  Routine calculations reveal
that $\pi(\bth, \mu \mid A, \by)$ is multivariate normal, and,
moreover, it is straightforward to construct a simple rejection
sampler to draw from the univariate density $\pi(A \mid \by)$
\citep[see, e.g.,][pp. 123-126]{JonesPHD}.  Given the choice between a
correlated sample from the Gibbs sampler and an iid sample, one would
probably choose that latter. On the other hand, the fact that these 
Gibbs samplers would probably not be used in practice doesn't render 
them easy to analyze. Indeed, it is still unknown whether the 
convergence rate of Rosenthal's Gibbs sampler remains bounded away 
from 1 as dimension grows, or if not, how quickly the rate approaches 
1 as dimension grows.

\vspace*{6mm}

\noindent {\bf \large Acknowledgment}.  The authors thank an anonymous
referee for their review and suggestions.

\vspace*{8mm}

\noindent{\LARGE \bf Appendix}
\appendix

\section{Proof of Proposition~\ref{pro:tc}}
\label{app:A}

We begin with two simple lemmas whose proofs are straightforward.

\begin{lemma}
  \label{lem:1}
  Suppose that $X$ is a continuous random variable with positive
  support, and let $s$ and $t$ be real numbers such that $1 \le s \le
  t < \infty$.  If $1 \le \mbox{E} X^t < \infty$, then $\mbox{E}X^s \le 
  1 + \mbox{E}X^t \le 2 \mbox{E} X^t$.
\end{lemma}

\begin{lemma}
  \label{lem:2}
  Suppose that $X \sim \chi^2_k (\phi)$ (non-central $\chi^2$ with $k$
  degrees of freedom, and non-centrality parameter $\phi$) where $k
  \ge 1$.  If $r \in \mathbb{N}$, then
\[
\mbox{E} \left[ X^r \right] \le C (k + r \phi)^r \;,
\]
where $C$ is a positive constant that does not depend on $\phi$ (but
may depend on $k$ and $r$).
\end{lemma}

\begin{proof}[Proof of Proposition~\ref{pro:tc}]

We begin with an overview of the argument, and then fill in the
details.  The goal is to show that
\[
\int_{\mathbb{R}^{+}} \int_{\mathbb{R}} \int_{\mathbb{R}^n} I(A, \mu,
\bth) \, \df \bth \, \df \mu \, \df A < \infty \;,
\]
where $I(A, \mu, \bth) = \pi(\bth \mid A, \mu) \, \pi(A, \mu \mid \bth)$.
We first show that
\[
\int_{\mathbb{R}^n} I(A, \mu, \bth) \, \df \bth
\]
can be bounded above by $h_1(A) \, h_2(\mu,A) \, h_3(A)$, where
$h_1(A)$ is the expectation of a function of a non-central $\chi^2$
random variable, $h_2(\mu,A)$ is a univariate normal density in the
variable $\mu$, and $h_3(A)$ is a simple function of $A$.  Hence,
\[
\int_{\mathbb{R}} \int_{\mathbb{R}^n} I(A, \mu, \bth) \, \df \bth \,
\df \mu \le h_1(A) \, h_3(A) \;.
\]
We then use Lemmas~\ref{lem:1} and \ref{lem:2} to show that $h_1(A)$
is bounded above by a constant.  Finally, a routine argument shows
that
\[
\int_{\mathbb{R}^{+}} h_3(A) \, \df A < \infty \;,
\]
which completes the argument.

We now provide the details.  Observe that 
\begin{align*}
  I(A, \mu, \bth) & = C_1 \left[ b + \frac{1}{2}
    \sum_{i=1}^{n} (\theta_i - \bar{\theta})^2 \right]^{a
    + \frac{n-1}{2}} A^{-a - \frac{n}{2} -1} \, \mbox{exp} \left\{
  -\frac{1}{A} \left[ b + \frac{1}{2} \sum_{i=1}^{n} (\theta_i -
    \bar{\theta})^2 \right] \right\} \\ & \times
  \mbox{exp} \left\{ -\frac{n}{2A} (\mu - \bar{\theta})^2
  \right\} \left( \frac{AV}{A+V} \right)^{-\frac{n}{2}} \mbox{exp}
  \left\{ -\frac{A+V}{2AV} \sum_{i=1}^{n} \left( \theta_i - \left(
  \frac{V\mu + A y_i}{A+V} \right) \right)^2 \right\} \;,
\end{align*}
where, throughout the proof, the $C_i$ are positive constants that do
not depend on $(\bth,\mu,A)$.  We begin by showing that
$\int_{\mathbb{R}^n} I(A, \mu, \bth) \,
d\bth$ can be expressed as an expectation with respect
to a multivariate normal distribution.  We have
\begin{align*}
  \frac{1}{A} &\sum_{i=1}^{n} \left(\theta_i -
  \bar{\theta} \right)^2 + \frac{n}{A} (\mu -
  \bar{\theta})^2 + \frac{A+V}{AV} \sum_{i=1}^{n} \left(
  \theta_i - \left( \frac{V\mu + A y_i}{A+V} \right) \right)^2 \\ &=
  \sum_{i=1}^{n} \theta_i^2 \left[ \frac{1}{A} + \frac{A+V}{AV}
    \right] + \frac{n}{A} \left( \bar{\theta}^2 -
  \bar{\theta}^2 + \mu^2 - 2 \mu
  \bar{\theta} \right) \\ &\hspace{4cm}+ \frac{A+V}{AV}
  \sum_{i=1}^{n} \left( \left( \frac{V\mu + A y_i}{A+V} \right)^2 - 2
  \theta_i \left( \frac{V\mu + A y_i}{A+V} \right) \right) \\ &=
  \frac{A+2V}{AV} \sum_{i=1}^{n} \theta_i^2 + \frac{n}{A} \mu^2 -
  \frac{2}{AV} \sum_{i=1}^{n} \theta_i (2V\mu + Ay_i) +
  \frac{1}{AV(A+V)} \sum_{i=1}^{n} (V\mu + Ay_i)^2 \\ &=
  \frac{A+2V}{AV} \sum_{i=1}^{n} \left( \theta_i - \left( \frac{2V\mu
    + Ay_i}{A+2V} \right) \right)^2 + \frac{n}{A} \mu^2
  \\ &\hspace{4cm}+ \frac{\sum_{i=1}^{n} (V\mu + Ay_i)^2 }{AV(A+V)} -
  \frac{\sum_{i=1}^{n} (2V\mu + Ay_i)^2 }{AV(A+2V)} \;.
\end{align*}
Letting $G = \frac{A+2V}{AV} \sum_{i=1}^{n} \left( \theta_i - \left(
\frac{2V\mu + Ay_i}{A+2V} \right) \right)^2$, we have
\begin{align*}
  \frac{1}{A} &\sum_{i=1}^{n} \left(\theta_i -
  \bar{\theta} \right)^2 + \frac{n}{A} (\mu -
  \bar{\theta})^2 + \frac{A+V}{AV} \sum_{i=1}^{n} \left(
  \theta_i - \left( \frac{V\mu + A y_i}{A+V} \right) \right)^2 \\ &= G
  + \frac{1}{A} \left[ n \mu^2 + \frac{1}{V(A+V)} \left( nV^2 \mu^2 +
    2VAn\bar{y} \mu + A^2 \sum_{i=1}^{n} y_i^2 \right)
    \right. \\ &\hspace{4cm} \left. - \frac{1}{V(A+2V)} \left( 4nV^2
    \mu^2 + 4VAn\bar{y} \mu + A^2 \sum_{i=1}^{n} y_i^2 \right) \right]
  \\ &= G + \frac{nA}{(A+V)(A+2V)} \left( \mu^2 - 2 \mu \bar{y}
  \right) + \frac{A}{(A+V)(A+2V)} \sum_{i=1}^{n} y_i^2 \\ &= G +
  \frac{nA}{(A+V)(A+2V)} \left( \mu - \bar{y} \right)^2 +
  \frac{A}{(A+V)(A+2V)} \Delta \;,
\end{align*}
where $\Delta = \sum_{i=1}^{n} (y_i - \bar{y})^2$.

Now let $\bth_{A, \mu}$ denote the $n \times 1$ vector
whose $i$th entry is $\theta_{A, \mu, i} = \frac{2V\mu + A
  y_i}{A+2V}$.  Let $\mbox{E}_* \left[ f(\bth) \right]$
denote the expected value of a function $f(\bth)$ when
$\bth \sim \mbox{N} \big( \bth_{A, \mu},
\frac{AV}{A+2V} I_n \big)$.  Since $\frac{A}{(A+V)(A+2V)} \Delta >0$,
we have
\begin{align}
  \label{eq:1}
  \int_{\mathbb{R}^n} I(A, \mu, \bth) \,
  d\bth \le C_2 \, A^{-a -\frac{n}{2} - 1} &
  e^{-\frac{b}{A}} \left( \frac{A+V}{A+2V} \right)^{\frac{n}{2}}
  \mbox{E}_* \Bigg[ \bigg( b + \frac{1}{2} \bth^T
    \bigg( I - \frac{1}{n} J \bigg) \bth \bigg)^{a +
      \frac{n-1}{2}} \Bigg] \nonumber \\ & \times \mbox{exp} \left\{
  -\frac{nA}{2(A+V)(A+2V)} \left( \mu - \bar{y} \right)^2 \right\} \;.
\end{align}
Now, it follows from basic distribution theory that, if
$\bth \sim \mbox{N} \big( \bth_{A, \mu},
\frac{AV}{A+2V} I_n \big)$, then
\[
\frac{A+2V}{AV} \bth^T \left(I - \frac{1}{n} J \right)
\bth \sim \chi^2_{n-1} (\phi) \;,
\]
where the non-centrality parameter is given by
\[
\phi = \frac{A}{2V(A+2V)} \Delta \;.
\]
We deduce from this that the expectation on the right-hand side of
\eqref{eq:1} does not depend on $\mu$ (but does depend on $A$).
Hence, we have
\begin{align}
  \label{eq:2}
  \int_{\mathbb{R}} \int_{\mathbb{R}^n} I(A, \mu, \bth)
  \, d\bth \, d\mu & \le C_3 A^{-\big(a +\frac{n}{2}
    +\frac{3}{2} \big)} \, e^{-\frac{b}{A}}
  \frac{(A+V)^\frac{n+1}{2}}{(A+2V)^\frac{n-1}{2}} \, \mbox{E}_*
  \Bigg[ \bigg( b + \frac{1}{2} \bth^T \bigg( I -
    \frac{1}{n} J \bigg) \bth \bigg)^{a +
      \frac{n-1}{2}} \Bigg] \nonumber \\ & \le C_3 A^{-\big(a
    +\frac{n}{2} +\frac{3}{2} \big)} \, e^{-\frac{b}{A}} \, (A+V) \,
  \mbox{E}_* \Bigg[ \bigg( b + \frac{1}{2} \bth^T
    \bigg( I - \frac{1}{n} J \bigg) \bth \bigg)^{a +
      \frac{n-1}{2}} \Bigg] \;.
\end{align}
Let $N = \lceil a + \frac{n-1}{2} \rceil$, where $\lceil \cdot \rceil$
returns the smallest integer that exceeds the argument.  Since $n \ge
3$, $a + \frac{n - 1}{2} > 1$, and $N \ge 2$.  Now,
\[
  \mbox{E}_* \left[ \frac{A+2V}{AV} \bth^T \left(I - \frac{1}{n} J
    \right) \bth \right] = n - 1 + 2 \phi > 1 \;.
\]
Therefore, Jensen's inequality implies that
\[
\mbox{E}_* \Bigg[ \bigg( \frac{A+2V}{AV} \bth^T \bigg(I
  - \frac{1}{n} J \bigg) \bth \bigg)^{N} \Bigg] > 1 \;.
\]
Applying Lemma~\ref{lem:1} yields
\[
\mbox{E}_* \Bigg[ \bigg( \frac{A+2V}{AV} \bth^T \bigg(I
  - \frac{1}{n} J \bigg) \bth \bigg)^{a + \frac{n -
      1}{2}} \Bigg] \le 2 \mbox{E}_* \Bigg[ \bigg( \frac{A+2V}{AV}
  \bth^T \bigg(I - \frac{1}{n} J \bigg)
  \bth \bigg)^{N} \Bigg] \;.
\]
Now, using the fact that $A/(A+2V)<1$ and applying Lemma~\ref{lem:2}, we have
\begin{align*}
\mbox{E}_* \Bigg[ \bigg( \bth^T \bigg(I - \frac{1}{n} J
  \bigg) \bth \bigg)^{a + \frac{n - 1}{2}} \Bigg] & =
\mbox{E}_* \Bigg[ \bigg( \frac{AV}{A+2V} \frac{A+2V}{AV}
  \bth^T \bigg(I - \frac{1}{n} J \bigg)
  \bth \bigg)^{a + \frac{n - 1}{2}} \Bigg] \\ & \le 2
\bigg(\frac{AV}{A+2V}\bigg)^{a + \frac{n - 1}{2}} \mbox{E}_* \Bigg[
  \bigg( \frac{A+2V}{AV} \bth^T \bigg(I - \frac{1}{n} J
  \bigg) \bth \bigg)^{N} \Bigg] \\ & \le C_4
\bigg(\frac{A}{A+2V}\bigg)^{a + \frac{n - 1}{2}} \bigg( n - 1 + N
\frac{A}{2V(A+2V)} \Delta \bigg)^N \\ & \le C_5 \;.
\end{align*}
Since $(u+v)^p \le (2u)^p + (2v)^p$ whenever all three variables are
positive, we have
\[
  \mbox{E}_* \Bigg[ \bigg( b + \frac{1}{2} \bth^T
    \bigg( I - \frac{1}{n} J \bigg) \bth \bigg)^{a +
      \frac{n-1}{2}} \Bigg] \le (2b)^{a + \frac{n-1}{2}} + \mbox{E}_*
  \Bigg[ \bigg( \bth^T \bigg( I - \frac{1}{n} J \bigg)
    \bth \bigg)^{a + \frac{n-1}{2}} \Bigg] \le C_6 \;.
\]
Combining this with \eqref{eq:2}, we have
\[
  \int_{\mathbb{R}} \int_{\mathbb{R}^n} I(A, \mu, \bth)
  \, d\bth \, d\mu \le C_7 \, A^{-\big(a +\frac{n}{2}
    +\frac{3}{2} \big)} \, e^{-\frac{b}{A}} \, (A+V) \;.
\]
Finally, it's clear that
\[
  \int_{\mathbb{R}_{+}} A^{-\big(a +\frac{n}{2} +\frac{3}{2} \big)} \,
  (A+V) \, e^{-\frac{b}{A}} \, dA < \infty \;,
\]
and the proof is complete.
\end{proof}

\section{Proof of Proposition~\ref{pro:mcv}}
\label{app:B}

\begin{proof}
We begin with an overview of the argument, and then fill in the
details.  The goal is to show that
\[
\int_{\mathbb{R}^{+}} \int_{\mathbb{R}} \int_{\mathbb{R}^n}
\frac{I'(A, \mu, \bth)}{\omega^2(A,\mu)} \, \df \bth \, \df \mu \, \df
A < \infty \;,
\]
where $I'(A, \mu, \bth) = \pi^3(A, \mu \mid \bth) \, \pi(\bth \mid A,
\mu)$.  Arguments similar to those used in the proof of
Proposition~\ref{pro:tc} show that
\[
\int_{\mathbb{R}^n} I'(A, \mu, \bth) \, \df \bth \le h'_2(\mu,A)
\, h'_3(A) \;,
\]
where $h'_2(\mu,A)$ is a univariate normal density in the variable $\mu$
and $h'_3(A)$ is a simple function of $A$.  It is then shown that
\[
\frac{h'_2(\mu,A) \, h'_3(A)}{\omega^2(A,\mu)} \le h''_2(\mu,A) \,
h''_3(A) \;,
\]
where $h''_2(\mu,A)$ is another univariate normal density in the variable 
$\mu$ and $h''_3(A)$ is another simple function of $A$.  It follows that
\[
\int_{\mathbb{R}} \int_{\mathbb{R}^n} \frac{I'(A, \mu,
  \bth)}{\omega^2(A,\mu)} \, \df \bth \, \df \mu \le h''_3(A)
\;,
\]
and the result follows by establishing that 
\[
\int_{\mathbb{R}^{+}} h''_3(A) \, \df A < \infty \;.
\]
Here are the details.  Observe that $\pi^3(A, \mu \mid \bth) \pi(\bth
\mid A, \mu)$ is given by
\begin{align*}
 C_1 \Bigg[ b + & \frac{1}{2} \sum_{i=1}^{n} (\theta_i - \bar{\theta})^2
   \Bigg]^{3 \big( a + \frac{n-1}{2} \big)} A^{-3a - \frac{3n}{2} -3}
 \, \mbox{exp} \left\{ -\frac{3}{A} \left[ b + \frac{1}{2}
   \sum_{i=1}^{n} (\theta_i - \bar{\theta})^2 \right] \right\} \\ &
 \times \mbox{exp} \left\{ -\frac{3n}{2A} (\mu - \bar{\theta})^2
 \right\} \left( \frac{AV}{A+V} \right)^{-\frac{n}{2}} \mbox{exp}
 \left\{ -\frac{A+V}{2AV} \sum_{i=1}^{n} \left( \theta_i - \left(
 \frac{V\mu + A y_i}{A+V} \right) \right)^2 \right\} \;,
\end{align*}
where, throughout the proof, the $C_i$ are positive constants that do
not depend on $(\bth,\mu,A)$.  Calculations similar to
those in the proof of Proposition~\ref{pro:tc} show that
\begin{align*}
  \frac{3}{A} &\sum_{i=1}^{n} \left(\theta_i - \bar{\theta} \right)^2
  + \frac{3n}{A} (\mu - \bar{\theta})^2 + \frac{A+V}{AV}
  \sum_{i=1}^{n} \left( \theta_i - \left( \frac{V\mu + A y_i}{A+V}
  \right) \right)^2 \\ & = \frac{A+4V}{AV} \sum_{i=1}^{n} \left(
  \theta_i - \left( \frac{4V\mu + Ay_i}{A+4V} \right) \right)^2 +
  \frac{3nA}{(A+V)(A+4V)} \left( \mu - \bar{y} \right)^2 +
  \frac{3A}{(A+V)(A+4V)} \Delta \;.
\end{align*}
Now let $\bth_{A, \mu}$ denote the $n \times 1$ vector
whose $i$th entry is $\theta_{A, \mu, i} = \frac{4V\mu + A
  y_i}{A+4V}$.  Let $\mbox{E}_* \left[ f(\bth) \right]$
denote the expected value of a function $f(\bth)$ when
$\bth \sim \mbox{N} \big( \bth_{A, \mu},
\frac{AV}{A+4V} I_n \big)$.  Since $\frac{3A}{(A+V)(A+4V)} \Delta >0$,
we have
\begin{align*}
  \int_{\mathbb{R}^n} \pi^3(A, \mu \mid \bth) \, &
  \pi(\bth \mid A, \mu) \, d\bth \le C_2
  \, A^{-3a - \frac{3n}{2} -3} e^{-\frac{3b}{A}} \left(
  \frac{A+V}{A+4V} \right)^{\frac{n}{2}} \nonumber \\ & \times
  \mbox{E}_* \Bigg[ \bigg( b + \frac{1}{2} \bth^T
    \bigg( I - \frac{1}{n} J \bigg) \bth \bigg)^{3
      \big( a + \frac{n-1}{2} \big)} \Bigg] \exp \left\{
  -\frac{3nA}{2(A+V)(A+4V)} \left( \mu - \bar{y} \right)^2 \right\}
  \;.
\end{align*}
Now, it follows from basic distribution theory that, if
$\bth \sim \mbox{N} \big( \bth_{A, \mu},
\frac{AV}{A+4V} I_n \big)$, then
\[
\frac{A+4V}{AV} \bth^T \left(I - \frac{1}{n} J \right)
\bth \sim \chi^2_{n-1} (\phi) \;,
\]
where the non-centrality parameter is given by
\[
\phi = \frac{A}{2V(A+4V)} \Delta \;.
\]
An argument similar to one used in the proof of
Proposition~\ref{pro:tc} shows that
\[
\mbox{E}_* \Bigg[ \bigg( b + \frac{1}{2} \bth^T \bigg(
  I - \frac{1}{n} J \bigg) \bth \bigg)^{3 \big( a +
    \frac{n-1}{2} \big)} \Bigg] \le C_3 \;.
\]
Thus,
\[
  \int_{\mathbb{R}^n} \pi^3(A, \mu \mid \bth) \,
  \pi(\bth \mid A, \mu) \, \df \bth 
  \le C_4 \, A^{-3a - \frac{3n}{2} -3} e^{-\frac{3b}{A}} \left(
  \frac{A+V}{A+4V} \right)^{\frac{n}{2}} \mbox{exp} \left\{
  -\frac{3nA( \mu - \bar{y})^2}{2(A+V)(A+4V)} \right\} \;.
\]
It follows that 
\begin{align*}
& \int_{\mathbb{R}_+} \int_{\mathbb{R}} \int_{\mathbb{R}^n}
  \frac{\pi^3 (A, \mu \mid \bth) \pi
    (\bth \mid A, \mu) }{\omega^2 (A, \mu)} \, \df
  \bth \, \df \mu \, \df A \\ & \le C_5
  \int_{\mathbb{R}_+} \int_{\mathbb{R}} A^{-a - \frac{3n}{2} -2}
  e^{-\frac{b}{A}} \left( \frac{A+V}{A+4V} \right)^{\frac{n}{2}}
  (A+V)(A+4V) \, \exp \left\{ -\frac{nA( \mu -
    \bar{y})^2}{2(A+V)(A+4V)} \right\} \, \df \mu \, \df A \\ & =
  C_6 \int_{\mathbb{R}_+} A^{-a - \frac{3n}{2} - \frac{5}{2}}
  e^{-\frac{b}{A}} \left( \frac{A+V}{A+4V} \right)^{\frac{n}{2}}
  (A+V)^{\frac{3}{2}} (A+4V)^{\frac{3}{2}} \, \df A \\ & = C_6
  \int_{\mathbb{R}_+} A^{-a - \frac{3n}{2} - \frac{5}{2}}
  e^{-\frac{b}{A}} \left( \frac{A+V}{A+4V} \right)^{\frac{n-3}{2}} 
  (A+V)^{3} \, \df A \\ & \le C_6 \int_{\mathbb{R}_+} A^{-a
    - \frac{3n}{2} - \frac{5}{2}} e^{-\frac{b}{A}} (A+V)^3 \, \df A
  \\ & < \infty \;.
\end{align*}
\end{proof}

\section{Proof of Proposition~\ref{prop:block}}
\label{app:C}

We prove Proposition~\ref{prop:block} by utilizing
Theorem~\ref{thm:ms} and Proposition~\ref{prop:ol}.  In particular, we
first show that the hypothesis of Theorem~\ref{thm:ms} holds for our
chain.  It then follows from Theorem~\ref{thm:ms} that, if the
Wasserstein geometric rate of convergence of our chain goes to zero as
$n \rightarrow \infty$, then the TV geometric rate of convergence goes 
to zero as well. We then use Proposition~\ref{prop:ol} to show that the
Wasserstein geometric rate of convergence does indeed go to zero as $n
\rightarrow \infty$.

Recall that the Mtd associated with $K$ is given by
\[
k(\bet' \mid \bet) = \int_0^\infty \pi(\bet' \mid B, \by) \, \pi(B
\mid \bet, \by) \, \df B \;.
\]
The following result shows that the hypothesis of Theorem~\ref{thm:ms}
holds for our Markov chain.

\begin{lemma}
  \label{eq:MS}
  Assume that $n \ge 2$ and $r \ge 1$.  There exists a constant $c =
  c(a,b) < \infty$ such that, for all $\bet,\bet' \in
  \mathbb{R}^{n+1}$,
\[
\int_{\mathbb{R}^{n+1}} \Big| k(\bet'' \mid \bet) - k(\bet'' \mid
\bet') \Big| \, \df \bet'' \le c \, n \norm{\bet - \bet'} \;.
\]
\end{lemma}

\begin{proof}
We begin with an overview of the argument, and then fill in the
details.  First, it is shown that
\begin{equation}
  \label{eq:o1}
  \int_{\mathbb{R}^{n+1}} \Big| k(\bet'' \mid \bet) - k(\bet'' \mid
  \bet') \Big| \, \df \bet'' \le \int_0^\infty \Big| \pi(B \mid \bet,
  \by) - \pi(B \mid \bet', \by) \Big| \, \df B \;.
\end{equation}
Thus, we can work with an integral on $\mathbb{R}_{+}$ rather than an
integral on $\mathbb{R}^{n+1}$.  Recall that
\[
B \mid \bet, \by \sim \mbox{Gamma} \bigg( a + \frac{n}{2}, b + \frac{
  \sum_{i=1}^n \eta^2_i}{2} \bigg) \;.
\]
We provide a closed-form expression for the integral on the right-hand
side of \eqref{eq:o1}, and then it is shown that this expression is
bounded above by an explicit function of $\norm{\bet - \bet'}$, call
it $g(\norm{\bet - \bet'})$.  We then apply the mean value theorem to
the function $g(\cdot)$ to show that there exists a finite constant
$c$ such that
\[
\int_{\mathbb{R}^{n+1}} \Big| k(\bet'' \mid \bet) - k(\bet'' \mid
\bet') \Big| \, \df \bet'' \le c \, n \norm{\bet - \bet'}
\]
whenever $\norm{\bet - \bet'} < 1/n$.  Finally, the triangle
inequality is used to extend the result to pairs $(\bet,\bet')$ for
which $\norm{\bet - \bet'} \ge 1/n$.  This completes the argument.

Here are the details.  Note that
\begin{align}
  \label{eq:c4}
\int_{\mathbb{R}^{n+1}} \Big| k(\bet'' \mid \bet) - k(\bet'' \mid
\bet') \Big| \, \df \bet'' & = \int_{\mathbb{R}^{n+1}} \bigg|
\int_0^\infty \pi(\bet'' \mid B, \by) \Big[ \pi(B \mid \bet, \by) -
  \pi(B \mid \bet', \by) \Big] \, \df B \, \bigg| \, \df \bet''
\nonumber \\ & \le \int_{\mathbb{R}^{n+1}} \int_0^\infty \pi(\bet''
\mid B, \by) \, \Big| \pi(B \mid \bet, \by) - \pi(B \mid \bet', \by)
\Big| \, \df B \, \df \bet'' \nonumber \\ & = \int_0^\infty \Big|
\pi(B \mid \bet, \by) - \pi(B \mid \bet', \by) \Big| \, \df B \;.
\end{align}
Define $\delta(\bet, \bet') = S(\bet') - S(\bet)$, where, for $\bx =
(x_0, x_1, \dots, x_n)^T \in \mathbb{R}^{n+1}$, $S(\bx) := b +
\frac{1}{2} \sum_{i=1}^{n} x_i^2$.  (Note that $S(\bet)$ is free of
$\eta_0$).  Now
\begin{align*}
  \delta(\bet, \bet') & = \frac{1}{2} \sum_{i=1}^{n} \Big[ (\eta_i')^2
    - \eta_i^2 \Big] = \frac{1}{2} \sum_{i=1}^{n} \Big[ (\eta_i' -
    \eta_i)^2 - (\eta_i' - \eta_i)^2 + (\eta_i')^2 - \eta_i^2 \Big]
  \\ & = \sum_{i=1}^{n} \bigg[ \eta_i(\eta_i' - \eta_i) +
    \frac{(\eta_i' - \eta_i)^2}{2} \bigg] \le \bigg \{ \Big(
  \sum_{i=1}^{n} \eta_i (\eta_i' - \eta_i) \Big)^2 \bigg\}^{1/2} +
  \frac{1}{2} \norm{\bet - \bet'}^2 \\ & \le \bigg\{ \bigg(
  \sum_{i=1}^{n} \eta_i^2 \bigg) \norm{\bet - \bet'}^2 \bigg\}^{1/2} +
  \frac{1}{2} \norm{\bet - \bet'}^2 \le \sqrt{2 S(\bet)} \, \norm{\bet
    - \bet'} + \frac{1}{2} \norm{\bet - \bet'}^2 \;.
\end{align*}
WLOG, assume that $S(\bet') \ge S(\bet)$.  Define the point
\[
t_1 = \frac{\frac{n}{2} + a}{\delta(\bet, \bet')} \log \bigg[ 1 +
  \frac{\delta(\bet, \bet')}{S(\bet)} \bigg] \;.
\]
Observe that $\pi(B \mid \bet', \by) \ge \pi(B \mid \bet, \by)$ if and
only if
\[
\frac{S(\bet')^{a + n/2}}{\Gamma(a + n/2)} B^{a + n/2 - 1} e^{-B
  S(\bet')} \ge \frac{S(\bet)^{a + n/2}}{\Gamma(a + n/2)} B^{a + n/2 -
  1} e^{-B S(\bet)} \;,
\]
which happens if and only if
\[
\bigg[ \frac{S(\bet')}{S(\bet)} \bigg]^{a + n/2} \ge e^{B (S(\bet') -
  S(\bet))} \;,
\]
which happens if and only if $B \le t_1.$

We now use these results to bound the right-hand side of
\eqref{eq:c4}.  We have
\begin{align*}
\int_0^\infty \Big| & \pi(B \mid \bet, \by) - \pi(B \mid \bet', \by)
\Big| \, \df B \\ & = 2 \int_0^{t_1} \bigg[ \frac{S(\bet')^{a +
      n/2}}{\Gamma(a + n/2)} u^{a + n/2 - 1} e^{-u S(\bet')} -
  \frac{S(\bet)^{a + n/2}}{\Gamma(a + n/2)} u^{a + n/2 - 1} e^{-u
    S(\bet)} \bigg] \, \df u \\ & = 2 \int_0^{t_1} \frac{S(\bet)^{a +
    n/2}}{\Gamma(a + n/2)} u^{a + n/2 - 1} e^{-u S(\bet)} \bigg[
  \bigg( \frac{S(\bet')}{S(\bet)} \bigg)^{a + n/2} e^{u( S(\bet) -
    S(\bet'))} - 1 \bigg] \, \df u \\ & \le 2 \bigg[ \bigg(
  \frac{S(\bet')}{S(\bet)} \bigg)^{a + n/2} - 1 \bigg] \int_0^{t_1}
\frac{S(\bet)^{a + n/2}}{\Gamma(a + n/2)} u^{a + n/2 - 1} e^{-u
  S(\bet)} \, \df u \\ & \le 2 \bigg[ \bigg( \frac{S(\bet')}{S(\bet)}
  \bigg)^{a + n/2} - 1 \bigg] \int_0^\infty \frac{S(\bet)^{a +
    n/2}}{\Gamma(a + n/2)} u^{a + n/2 - 1} e^{-u S(\bet)} \, \df u
\\ & \le 2 \bigg[ \bigg( \frac{S(\bet) + \delta(\bet,\bet')}{S(\bet)}
  \bigg)^{a + n/2} - 1 \bigg] \\ & = 2 \bigg[ \bigg( 1 +
  \frac{\delta(\bet,\bet')}{S(\bet)} \bigg)^{a + n/2} - 1 \bigg] \\ &
\le 2 \Bigg[ \bigg( 1 + \sqrt{\frac{2}{S(\bet)}} \, \norm{\bet -
    \bet'} + \frac{\norm{\bet - \bet'}^2}{2S(\bet)} \bigg)^{a + n/2} -
  1 \Bigg] \\ & \le 2 \Bigg[ \bigg( 1 + \sqrt{\frac{2}{b}} \,
  \norm{\bet - \bet'} + \frac{\norm{\bet - \bet'}^2}{2b} \bigg)^{a +
    n/2} - 1 \Bigg] \;.
\end{align*}
Now consider the function $f:[0,\infty) \rightarrow (0,\infty)$ given
  by
\[
f(v) = \bigg( 1 + \sqrt{\frac{2}{b}} v + \frac{v^2}{2b} \bigg)^{a +
  n/2} \;.
\]
Note that 
\[
f'(v) = \Big( a + \frac{n}{2} \Big) \bigg( 1 + \sqrt{\frac{2}{b}} v +
\frac{v^2}{2b} \bigg)^{a + n/2 -1} \bigg( \sqrt{\frac{2}{b}} +
\frac{v}{b} \bigg) \;.
\]
If $\norm{\bet - \bet'} \le \frac{1}{n}$, then for any $c \in [0,
  \norm{\bet - \bet'}]$, we have
\[
\frac{1}{n} f'(c) \le w(n) := \frac{1}{n} \Big( a + \frac{n}{2} \Big)
\bigg( 1 + \sqrt{\frac{2}{b}} \frac{1}{n} + \frac{1}{2bn^2} \bigg)^{a
  + n/2 -1} \bigg( \sqrt{\frac{2}{b}} + \frac{1}{bn} \bigg) \;.
\]
Some analysis reveals that $\lim_{n \rightarrow \infty} w(n) <
\infty$.  Thus, as $n \to \infty$, $\frac{1}{n} f'(c)$ is bounded
above.  Let $c = \underset{n \ge 2}{\max} \; w(n)$, which is finite.
Then, by the mean value theorem, if $\norm{\bet - \bet'} \le
\frac{1}{n}$, we have
\[
\int_{\mathbb{R}^{n+1}} \Big| k(\bet'' \mid \bet) - k(\bet'' \mid
\bet') \Big| \, \df \bet'' \le 2 c n \norm{\bet - \bet'} \;.
\]
We now extend the result to pairs $(\bet,\bet')$ such that $\norm{\bet
  - \bet'} > \frac{1}{n}$.  Divide the vector $\bet - \bet'$ into
segments whose lengths are less than $\frac{1}{n}$.  In particular,
let $\{\bet_{(i)}\}_{i=0}^N$ be points in $\mathbb{R}^{n+1}$ such that
\[
\sum_{i=1}^N \big( \bet_{(i)} - \bet_{(i - 1)} \big) = \bet - \bet' \;
\]
where $\bet_{(0)} = \bet'$, $\bet_{(N)} = \bet$, $\bet_{(i)} -
\bet_{(i - 1)}$ has the same direction as $\bet - \bet'$, and
$\norm{\bet_{(i)} - \bet_{(i - 1)}} < \frac{1}{n}$.  Then, by what
have already shown, we have
\begin{align*}
   \int_{\mathbb{R}^{n+1}} \Big| k(\bet'' \mid \bet) - k(\bet'' \mid
   \bet') \Big| \, \df \bet'' & \le \sum_{i=1}^{N}
   \int_{\mathbb{R}^{n+1}} \Big| k(\bet'' \mid \bet_{(i)}) - k(\bet''
   \mid \bet_{(i-1)}) \Big| \, \df \bet'' \\ & \le 2 c n
   \sum_{i=1}^{N} \norm{\bet_{(i)} - \bet_{(i-1)}} \\ & = 2 c n
   \norm{\bet - \bet'} \;.
\end{align*}

\end{proof}

\begin{proof}[Proof of Proposition~\ref{prop:block}]

With Lemma~\ref{eq:MS} in hand, it suffices to show that the
Wasserstein geometric rate of convergence of our chain goes to zero as
$n \rightarrow \infty$.  This is accomplished using
Proposition~\ref{prop:ol}, which requires that we bound $\mbox{E} \,
\norm{f(\bet)-f(\bet')}$, where $f(\bet)$ is defined in
Subsection~\ref{ssec:block}.  Now, Lemma~\ref{lem:qin} and Jensen's
inequality yield
\begin{align}
  \label{eq:c3}
    \mbox{E} \, \norm{f(\bet) - f(\bet')} & \le \underset{t \in
      [0,1]}{\sup} \mbox{E} \, \Big \lVert \frac{\df}{\df t} f(\bet +
    t(\bet' - \bet)) \Big \rVert \nonumber \\ & = \underset{t \in
      [0,1]}{\sup} \mbox{E} \sqrt{ \sum_{i=0}^{n} \bigg(
      \frac{\df}{\df t} \tilde{\eta}_i^{(\bet + t(\bet' - \bet))}
      \bigg)^2} \nonumber \\ & \le \underset{t \in [0,1]}{\sup} \sqrt{
      \mbox{E} \sum_{i=0}^{n} \bigg( \frac{\df}{\df t}
      \tilde{\eta}_i^{(\bet + t(\bet' - \bet))} \bigg)^2} \;.
\end{align}
The rest of the proof is simply brute force analysis of
\[
\mbox{E} \bigg[ \bigg(\frac{\df}{\df t} \tilde{\eta}_i^{(\bet +
    t(\bet' - \bet))} \bigg)^2 \bigg] \;,
\]
for $i=0,1,\dots,n$.  Henceforth, we shall abbreviate using $\bal =
\bet' - \bet$, so that $\bet + t(\bet' - \bet) = \bet + t \bal$.  We
begin by calculating $\frac{\df}{\df t} \tilde{B}^{(\bet + t \bal)}$.
We have
\begin{align*}
\frac{\df}{\df t} \tilde{B}^{(\bet + t\bal)} & = \frac{\df}{\df t}
\frac{J}{b + \frac{1}{2} \sum_{i=1}^{n} (\eta_i + t \alpha_i)^2} \\ &
= - \frac{J}{\big[ b + \frac{1}{2} \sum_{i=1}^{n} (\eta_i + t
    \alpha_i)^2 \big]^2} \sum_{i=1}^{n} (\eta_i + t \alpha_i) \alpha_i
\\ & = -\frac{\left( \tilde{B}^{(\bet + t\bal)} \right)^2}{J}
\sum_{i=1}^{n} (\eta_i + t \alpha_i) \alpha_i \;.
\end{align*}
Thus, by Cauchy-Schwarz, we can see that
\begin{align}
  \label{eq:c1}
\bigg( \frac{\df}{\df t} \tilde{B}^{(\bet + t\bal)} \bigg)^2 & =
\frac{\big( \tilde{B}^{(\bet + t\bal)} \big)^4}{J^2} \bigg[
  \sum_{i=1}^{n} (\eta_i + t \alpha_i) \alpha_i \bigg]^2 \le
\frac{\big( \tilde{B}^{(\bet + t\bal)} \big)^4}{J^2} \bigg[
  \sum_{i=1}^{n} (\eta_i + t \alpha_i)^2 \bigg] \norm{\bal}^2
\nonumber \\ & = \frac{2 \big( \tilde{B}^{(\bet + t\bal)}
  \big)^4}{J^2} \bigg[ \frac{1}{2} \sum_{i=1}^{n} (\eta_i + t
  \alpha_i)^2 \bigg] \norm{\bal}^2 \le \frac{2 \big( \tilde{B}^{(\bet
    + t\bal)} \big)^3}{J} \norm{\bal}^2 \;.
\end{align}
Next, observe that 
\begin{equation}
  \label{eq:c2}
  \frac{\df}{\df t} \tilde{\eta}_0^{(\bet + t\bal)} = \frac{\df}{\df
    t} \sqrt{\frac{\tilde{B}^{(\bet + t\bal)} + r U}{r
      \tilde{B}^{(\bet + t\bal)} U}} N_0 = \frac{N_0}{2} \sqrt{\frac{r
      \tilde{B}^{(\bet + t\bal)} U}{\tilde{B}^{(\bet + t\bal)} + r U}}
  \, \Bigg[-\frac{\frac{\df}{\df t} \tilde{B}^{(\bet +
        t\bal)}}{\big(\tilde{B}^{(\bet + t\bal)} \big)^2} \Bigg] \;.
\end{equation}
Hence, using \eqref{eq:c1}, we have
\begin{align*}
\bigg( \frac{\df}{\df t} \tilde{\eta}_0^{(\bet + t\bal)} \bigg)^2 & =
\frac{N^2_0}{4} \bigg(\frac{r \tilde{B}^{(\bet + t\bal)}
  U}{\tilde{B}^{(\bet + t\bal)} + r U} \bigg) \,
\frac{1}{\big(\tilde{B}^{(\bet + t\bal)} \big)^4} \Big( \frac{\df}{\df
  t} \tilde{B}^{(\bet + t\bal)} \Big)^2 \\ & \le \frac{N^2_0}{2}
\frac{rU}{J(\tilde{B}^{(\bet + t\bal)} +rU)} \norm{\bal}^2 \\ &
\le \frac{N^2_0}{2J} \norm{\bal}^2 \;.
\end{align*}
It follows that
\begin{equation*}
\mbox{E} \bigg[ \bigg( \frac{\df}{\df t} \tilde{\eta}_0^{(\bet +
    t\bal)} \bigg)^2 \bigg] \le \frac{1}{2a+n-2} \norm{\bal}^2 \;.
\end{equation*}
Next, we calculate $\frac{\df}{\df t} \tilde{\eta}_i^{(\bet + t\bal)}$
for $i = 1,\dots,n$.  We have
\begin{align*}
\frac{\df}{\df t} \tilde{\eta}_{i}^{(\bet +t \bal)} & = \frac{\df}{\df
  t} \Bigg[ \frac{r U}{\tilde{B}^{(\bet+t \bal)} + r U} \bigg(
  \bar{y}_i - \frac{\tilde{\eta}_{0}^{(\bet+t \bal)}}{\sqrt{n}} \bigg)
  + \sqrt{\frac{1}{\tilde{B}^{(\bet+t \bal)} + r U }}N_i \Bigg] \\ & =
\frac{r U}{(\tilde{B}^{(\bet+t \bal)} + r U)^2} \bigg(
\frac{\tilde{\eta}_{0}^{(\bet+t \bal)}}{\sqrt{n}} - \bar{y}_i \bigg)
\Big( \frac{\df}{\df t} \tilde{B}^{(\bet + t\bal)} \Big) - \frac{r
  U}{\sqrt{n}(\tilde{B}^{(\bet+t \bal)} + r U)} \bigg( \frac{\df}{\df
  t} \tilde{\eta}_0^{(\bet +t \bal)} \bigg) \\ & \hspace*{10mm} -
\frac{N_i}{2(\tilde{B}^{(\bet+t \bal)} + r U)^{\frac{3}{2}}} \Big(
\frac{\df}{\df t} \tilde{B}^{(\bet + t\bal)} \Big) \\ & = T_{1,i} +
T_{2,i} + T_3 \;,
\end{align*}
where 
\[
T_{1,i} = \frac{r U}{(\tilde{B}^{(\bet+t \bal)} + r U)^2} (\bar{y} -
\bar{y}_i) \Big( \frac{\df}{\df t} \tilde{B}^{(\bet + t\bal)} \Big)
\;,
\]
\[
T_{2,i} = - \frac{N_i}{2(\tilde{B}^{(\bet+t \bal)} + r
  U)^{\frac{3}{2}}} \Big( \frac{\df}{\df t} \tilde{B}^{(\bet + t\bal)}
\Big) \;,
\]
and
\[
T_3 = \frac{\sqrt{r U} N_0}{\sqrt{n \tilde{B}^{(\bet+t \bal)}}
  (\tilde{B}^{(\bet+t \bal)} + r U)^{\frac{3}{2}}} \Big(
\frac{\df}{\df t} \tilde{B}^{(\bet + t\bal)} \Big) - \frac{r U
}{\sqrt{n}(\tilde{B}^{(\bet+t \bal)} + r U)} \bigg( \frac{\df}{\df t}
\tilde{\eta}_0^{(\bet +t \bal)} \bigg) \;.
\]
Equation \eqref{eq:c2} shows that $\frac{\df}{\df t}
\tilde{\eta}_0^{(\bet +t \bal)}$ has a factor of $N_0$.  Thus,
$T_{1,i}$ has no normal terms in it, $T_{2,i} = c N_i$, and $T_3 = d
N_0$, where $c$ and $d$ do not have any normal terms in them.  Since
all the normal random variables are independent of each other (and of
$J$), we have
\[
\mbox{E} \bigg[ \bigg( \frac{\df}{\df t} \tilde{\eta}_i^{(\bet + t
    \bal)} \bigg)^2 \bigg] = \mbox{E} \big[ (T_{1,i} + T_{2,i} +
  T_3)^2 \big] = \mbox{E} \big[T_{1,i}^2\big] + \mbox{E}
\big[T_{2,i}^2\big] + \mbox{E} \big[T_3^2\big] \;.
\]
Let $S(\bx) = b + \frac{1}{2} \sum_{i=1}^{n} x_i^2$ for any $\bx \in
\mathbb{R}^n$.  Letting $\Delta' = \sum_{i=1}^n (\bar{y}_i -
\bar{y})^2$, we have
\begin{align*}
\sum_{i=1}^n \mbox{E} \big[T_{1,i}^2\big] & = \Delta' \mbox{E} \bigg[
  \frac{(r U)^2}{(\tilde{B}^{(\bet+t \bal)} + r U)^4} \Big(
  \frac{\df}{\df t} \tilde{B}^{(\bet + t\bal)} \Big)^2 \bigg] \\ & \le
\Delta' \mbox{E} \bigg[ \frac{1}{(r U)^2} \Big( \frac{\df}{\df t}
  \tilde{B}^{(\bet + t\bal)} \Big)^2 \bigg] \\ & \le \frac{2
  \Delta'}{n} \mbox{E} \bigg[ \frac{n \big( \tilde{B}^{(\bet +
      t\bal)} \big)^3}{J(r U)^2} \bigg] \norm{\bal}^2 \\ & = \frac{2
  \Delta'}{n} \mbox{E} \bigg[ \frac{n J^2}{(r U)^2 S^3(\bet + t\bal)}
  \bigg] \norm{\bal}^2 \\ & = \frac{\Delta'}{n} \frac{n
  (2a+n)(2a+n+2)}{2(r U)^2 S^3(\bet + t\bal)} \norm{\bal}^2 \\ & \le
\frac{\Delta'}{n} \frac{n (2a+n)(2a+n+2)}{2(r U)^2 b^3}
\norm{\bal}^2 \;.
\end{align*}
Now
\begin{align*}
\sum_{i=1}^n \mbox{E} \big[T_{2,i}^2\big] & = \frac{n}{4} \mbox{E}
\bigg[ \frac{1}{(\tilde{B}^{(\bet+t \bal)} + r U)^3} \Big(
  \frac{\df}{\df t} \tilde{B}^{(\bet + t\bal)} \Big)^2 \bigg] \\ & \le
\frac{n}{2} \mbox{E} \bigg[ \frac{\big(\tilde{B}^{(\bet + t\bal)}
    \big)^3}{J(\tilde{B}^{(\bet+t \bal)} + r U)^3} \bigg]
\norm{\bal}^2 \\ & = \frac{n}{2} \mbox{E} \bigg[
  \frac{\big(\tilde{B}^{(\bet + t\bal)} \big)^2}{S(\bet + t\bal)
    (\tilde{B}^{(\bet+t \bal)} + r U)^3} \bigg] \norm{\bal}^2 \\ & \le
\frac{n}{2brU} \norm{\bal}^2 \;.
\end{align*}
Finally, using the fact that $(u + v)^2 \le 2u^2 + 2v^2$, we have
\begin{align*}
\sum_{i=1}^n \mbox{E} \big[T_3^2\big] &\le n \mbox{E} \bigg[ \frac{2 r
    U N^2_0}{n \tilde{B}^{(\bet+t \bal)} (\tilde{B}^{(\bet+t \bal)} +
    r U)^3} \Big( \frac{\df}{\df t} \tilde{B}^{(\bet + t\bal)} \Big)^2
  + \frac{2(r U)^2 }{n(\tilde{B}^{(\bet+t \bal)} + r U)^2} \bigg(
  \frac{\df}{\df t} \tilde{\eta}_0^{(\bet +t \bal)} \bigg)^2 \bigg]
\\ & \le \mbox{E} \bigg[ \frac{4 r U N^2_0 \big( \tilde{B}^{(\bet +
      t\bal)} \Big)^3}{J \tilde{B}^{(\bet+t \bal)} (\tilde{B}^{(\bet+t
      \bal)} + r U)^3} + \frac{(r U)^2 N^2_0}{J(\tilde{B}^{(\bet+t
      \bal)} + r U)^2} \bigg] \norm{\bal}^2 \\ & = \mbox{E} \bigg[
  \frac{4 r U \big( \tilde{B}^{(\bet + t\bal)} \Big)^2}{J
    (\tilde{B}^{(\bet+t \bal)} + r U)^3} + \frac{(r
    U)^2}{J(\tilde{B}^{(\bet+t \bal)} + r U)^2} \bigg] \norm{\bal}^2
\\ & \le \norm{\bal}^2 \, \mbox{E} \big[ 5J^{-1} \big] \\ & =
\frac{10}{2a + n - 2} \norm{\bal}^2 \;.
\end{align*}
Therefore, combining \eqref{eq:c3} with the bounds developed above, we
have
\[
\mbox{E} \, \norm{f(\bet) - f(\bet')} \le \underset{t \in [0,1]}{\sup}
\sqrt{ \mbox{E} \sum_{i=0}^{n} \bigg( \frac{\df}{\df t}
  \tilde{\eta}_i^{(\bet + t(\bet' - \bet))} \bigg)^2} \le \gamma_{n,r}
\norm{\bet -\bet'} \;,
\]
where 
\begin{align*}
\gamma_{n,r} & = \gamma_{n,r}(\by,U,a,b) \\ & =
\sqrt{\frac{\Delta'}{n} \frac{n (2a+n)(2a+n+2)}{2(r U)^2 b^3} +
  \frac{n}{2brU} + \frac{11}{2a + n - 2}} \;
\end{align*}
Under (A1) and (A2), $\gamma_{n,r} \rightarrow 0$ as $n \rightarrow
\infty$.  Thus, for all large $n$, $\gamma_{n,r}<1$, and
Proposition~\ref{prop:ol} implies that, for every $\bet \in
\mathbb{R}^{n+1}$, we have
\[
d_{\mbox{\scriptsize{W}}}(K^m_{\bet},\Pi) \le 
\frac{c(\bet)}{1-\gamma_{n,r}} \, \gamma^m_{n,r} \;,
\]
where $c(\bet) = c(\bet;n,r,\by,U,a,b)$.  The proof is now complete.
\end{proof}

\section{Proof of Proposition~\ref{prop:shrink}}
\label{app:D}

The proof of Proposition~\ref{prop:shrink} is very similar to the
proof of Proposition~\ref{prop:block}.  Recall that the Mtd associated
with $K$ is given by
\[
k(\bbe' \mid \bbe) = \int_0^\infty \int_{\mathbb{R}} \pi(\bbe' \mid
\mu, B, \by) \, \pi(B \mid \bbe, \by) \, \pi(\mu \mid \bbe, \by) \,
\df \mu \, \df B \;.
\]
The following result shows that the hypothesis of Theorem~\ref{thm:ms}
holds for our chain.

\begin{lemma}
  \label{eq:MS2}
  Assume that $n \ge 2$ and $r \ge 1$.  There exists a constant $c =
  c(a,b,U) < \infty$ such that, for all $\bbe,\bbe' \in \mathbb{R}^n$,
\[
\int_{\mathbb{R}^n} \Big| k(\bbe'' \mid \bbe) - k(\bbe'' \mid \bbe')
\Big| \, \df \bbe'' \le c \, (n + \sqrt{r}) \, \norm{\bbe - \bbe'} \;.
\]
\end{lemma}

\begin{proof}
We begin by noting that
\begin{align}
  \label{eq:d0}
  \int_{\mathbb{R}^n} & \Big| k(\bbe'' \mid \bbe) - k(\bbe'' \mid
  \bbe') \Big| \, \df \bbe'' \nonumber \\ & \le \int_{\mathbb{R}^n}
  \int_{\mathbb{R}} \int_0^\infty \pi(\bbe'' \mid B, \mu, \by) \,
  \Big| \pi(B \mid \bbe, \by) \, \pi(\mu \mid \bbe, \by) - \pi(B \mid
  \bbe', \by) \, \pi(\mu \mid \bbe', \by) \Big| \, \df \mu \, \df B \,
  \df \bbe'' \nonumber \\ & = \int_{\mathbb{R}} \int_0^\infty \Big|
  \pi(B \mid \bbe, \by) \, \pi(\mu \mid \bbe, \by) - \pi(B \mid \bbe',
  \by) \, \pi(\mu \mid \bbe', \by) \Big| \, \df \mu \, \df B \nonumber
  \\ & \le \int_{\mathbb{R}} \int_0^\infty \Big| \pi(B \mid \bbe, \by)
  \, \pi(\mu \mid \bbe, \by) - \pi(B \mid \bbe', \by) \, \pi(\mu \mid
  \bbe, \by) \Big| \, \df \mu \, \df B \nonumber \\ & \hspace*{20mm} +
  \int_{\mathbb{R}} \int_0^\infty \Big| \pi(B \mid \bbe', \by) \,
  \pi(\mu \mid \bbe, \by) - \pi(B \mid \bbe', \by) \, \pi(\mu \mid
  \bbe', \by) \Big| \, \df \mu \, \df B \nonumber \\ & =
  \int_0^\infty \big| \pi(B \mid \bbe, \by) - \pi(B \mid \bbe', \by)
  \big| \, \df B + \int_{\mathbb{R}} \big| \pi(\mu \mid \bbe, \by) -
  \pi(\mu \mid \bbe', \by) \big| \, \df \mu \;.
\end{align}
Arguments similar to those used in the proof of Lemma~\ref{eq:MS} can
be used to show that
\begin{equation}
  \label{eq:d1}
  \int_0^\infty \big| \pi(B \mid \bbe, \by) - \pi(B \mid \bbe', \by)
  \big| \, \df B \le c n \norm{\bbe - \bbe'} \;,
\end{equation}
where $c = c(a,b) < \infty$ is a constant.  We now go to work on
$\int_{\mathbb{R}} \big| \pi(\mu \mid \bbe, \by) - \pi(\mu \mid \bbe',
\by) \big| \, \df \mu$.  It's easy to show that
\[
\int_{\mathbb{R}} \big| \pi(\mu \mid \bbe, \by) - \pi(\mu \mid \bbe',
\by) \big| \, \df \mu = \int_{-\infty}^{\infty} \frac{1}{\sqrt{2 \pi}}
\big| e^{-u^2/2} - e^{-(u-m)^2/2} \big| \, \df u \;,
\]
where
\[
m = \frac{n r U}{\sqrt{n r U + z}} ( \bar{\beta} - \bar{\beta'}) \;.
\]
Assume that $\bar{\beta} \ge \bar{\beta}'$, so $m \ge 0$.  A
straightforward calculation shows that the term inside the absolute
value is non-negative if and only if $u \le m/2$.  Therefore,
\begin{align*}
\int_{\mathbb{R}} \big| \pi(\mu \mid \bbe, \by) - \pi(\mu \mid \bbe',
\by) \big| \, \df \mu & = 2 \int_{-\infty}^{m/2} \frac{1}{\sqrt{2
    \pi}} \Big[ e^{-\frac{1}{2} u^2} - e^{-\frac{1}{2} (u-m)^2} \Big]
\, \df u \\ & = 2 \bigg[ \Phi \Big( \frac{m}{2} \Big) - \Phi \Big( \!
  -\frac{m}{2} \Big) \bigg] \;,
\end{align*}
where $\Phi(\cdot)$ denotes the standard normal cdf.  Similar
consideration of the case $\bar{\beta} < \bar{\beta}'$ leads to the
following:
\[
\int_{\mathbb{R}} \big| \pi(\mu \mid \bbe, \by) - \pi(\mu \mid \bbe',
\by) \big| \, \df \mu = 2 \bigg[ \Phi \bigg( \frac{|m|}{2} \bigg) -
  \Phi \bigg( \! -\frac{|m|}{2} \bigg) \bigg] = 4 \bigg[ \frac{1}{2} -
  \Phi \bigg( \! -\frac{|m|}{2} \bigg) \bigg] \;.
\]
By the mean value theorem, there exists $d \in [-|m|/2,0]$ such that 
\[
4 \bigg[ \frac{1}{2} - \Phi \bigg( \! -\frac{|m|}{2} \bigg) \bigg] = 4
\Phi'(d) \frac{|m|}{2} \le \frac{2|m|}{\sqrt{2 \pi}} \;.
\]
By Cauchy-Schwarz, we have $|\bar{\beta} - \bar{\beta'}| \le
\norm{\bbe - \bbe'}/\sqrt{n}$, and it follows that
\begin{equation}
  \label{eq:d2}
  \int_{\mathbb{R}} \big| \pi(\mu \mid \bbe, \by) - \pi(\mu \mid
  \bbe', \by) \big| \, \df \mu \le \frac{2 n r U}{\sqrt{2 \pi (n r U +
      z)}} |\bar{\beta} - \bar{\beta'}| \le \sqrt{\frac{2U}{\pi}} \,
  \sqrt{r} \, \norm{\bbe - \bbe'} \;.
\end{equation}
Combining \eqref{eq:d0}, \eqref{eq:d1}, and \eqref{eq:d2} yields the
result.
\end{proof}

\begin{proof}[Proof of Proposition~\ref{prop:shrink}]

With Lemma~\ref{eq:MS2} in hand, it suffices to show that the
Wasserstein geometric rate of convergence of our chain goes to zero as
$n \rightarrow \infty$.  We accomplish this via
Proposition~\ref{prop:ol}, which requires that we bound $\mbox{E} \,
\norm{f(\bbe)-f(\bbe')}$, where $f(\bbe)$ is defined in
Subsection~\ref{ssec:shrink}.  As in the proof of
Proposition~\ref{prop:block}, Lemma~\ref{lem:qin} and Jensen's
inequality yield
\begin{equation}
  \label{eq:d3}
    \mbox{E} \, \norm{f(\bbe) - f(\bbe')} \le \underset{t \in
      [0,1]}{\sup} \sqrt{ \mbox{E} \sum_{i=1}^{n} \bigg(
      \frac{\df}{\df t} \tilde{\beta}_i^{(\bbe + t(\bbe' - \bbe))}
      \bigg)^2} \;.
\end{equation}
The rest of the proof is simply brute force analysis of
\[
\mbox{E} \bigg[ \bigg( \frac{\df}{\df t} \tilde{\beta}_i^{(\bbe +
    t(\bbe' - \bbe))} \bigg)^2 \bigg] \;,
\]
for $i=1,\dots,n$.  Henceforth, we shall abbreviate using $\bal =
\bbe' - \bbe$, so that $\bbe + t(\bbe' - \bbe) = \bbe + t \bal$.
Calculations similar to those used in the proof of
Proposition~\ref{prop:block} show that
\begin{equation*}
\bigg( \frac{\df}{\df t} \tilde{B}^{(\bbe + t\bal)} \bigg)^2 \le
\frac{2 \big( \tilde{B}^{(\bbe + t\bal)} \big)^3}{J} \norm{\bal}^2 \;.
\end{equation*}
Now, plugging in the value of $\tilde{\mu}^{(\bbe+ t\bal)}$ and
rearranging yields
\begin{align*}
\tilde{\beta}_{i}^{(\bbe + t\bal)} & = \frac{r U}{\tilde{B}^{(\bbe+
    t\bal)} + r U} \big( \bar{y}_i - \tilde{\mu}^{(\bbe+ t\bal)} \big)
+ \frac{N_i}{\sqrt{\tilde{B}^{(\bbe+ t\bal)} + r U }} \\ & = \frac{r U
  (\bar{y}_i - \bar{y})}{\tilde{B}^{(\bbe+ t\bal)} + r U} + \frac{n
  r^2 U^2 (\bar{\beta} + t \bar{\alpha})}{\big(\tilde{B}^{(\bbe+
    t\bal)} + r U \big) (nrU+z)} - \frac{r U z (w -
  \bar{y})}{\big(\tilde{B}^{(\bbe+ t\bal)} + r U \big) (nrU+z)}
\\ & \hspace{35mm} - \frac{r U N_0}{\big(\tilde{B}^{(\bbe+ t\bal)} + r
  U \big) \sqrt{nrU+z}} + \frac{N_i}{\sqrt{\tilde{B}^{(\bbe+ t\bal)} +
    r U }} \;.
\end{align*}
We then have
\begin{align*}
\frac{\df}{\df t} & \tilde{\beta}_{i}^{(\bbe + t\bal)} = - \frac{r U
  (\bar{y}_i - \bar{y})}{\big( \tilde{B}^{(\bbe+ t\bal)} + r U
  \big)^2} \bigg( \frac{\df}{\df t} \tilde{B}^{(\bbe + t\bal)} \bigg)
+ \frac{n r^2 U^2 \bar{\alpha}}{\big(\tilde{B}^{(\bbe+ t\bal)} + r U
  \big) (nrU+z)} \\ & - \frac{n r^2 U^2 (\bar{\beta} + t\bar{\alpha}
  )}{\big(\tilde{B}^{(\bbe+ t\bal)} + r U \big)^2 (nrU+z)} \bigg(
\frac{\df}{\df t} \tilde{B}^{(\bbe + t\bal)} \bigg) + \frac{r U z (w -
  \bar{y})}{\big(\tilde{B}^{(\bbe+ t\bal)} + r U \big)^2 (nrU+z)}
\bigg( \frac{\df}{\df t} \tilde{B}^{(\bbe + t\bal)} \bigg) \\ & +
\frac{r U N_0}{\big(\tilde{B}^{(\bbe+ t\bal)} + r U \big)^2
  \sqrt{nrU+z}} \bigg( \frac{\df}{\df t} \tilde{B}^{(\bbe + t\bal)}
\bigg) - \frac{N_i}{2 \big(\tilde{B}^{(\bbe+ t\bal)} + r U
  \big)^{3/2}} \bigg( \frac{\df}{\df t} \tilde{B}^{(\bbe + t\bal)}
\bigg) \;.
\end{align*}
Denote the right-hand side as $\sum_{j=1}^4 a_j + a_5 N_0 + a_6 N_i$.
Then, since $\{N_i\}_{i=0}^n$ are iid standard normal, we have
\begin{align*}
  \mbox{E} \Bigg[ \bigg( \sum_{j=1}^4 a_j + a_5 N_0 + a_6 N_i\bigg)^2
    \Bigg] & = \bigg( \sum_{j=1}^4 a_j \bigg)^2 + a^2_5 + a^2_6 \\ &
  \le 2a^2_1 + 4a^2_2 + 8a^2_3 + 8a^2_4 + a^2_5 + a^2_6 \;,
\end{align*}
where we have used the fact that $(u+v)^2 \le 2u^2 + 2v^2$ three
times.  Letting $\Delta' = \sum_{i=1}^n (\bar{y}_i - \bar{y})^2$, we
have
\begin{align*}
\mbox{E} \Bigg[ \sum_{i=1}^n \bigg( & \frac{\df}{\df t}
  \tilde{\beta}_{i}^{(\bbe + t\bal)} \bigg)^2 \Bigg] \le \mbox{E}
\Bigg[ \frac{4 r^2 U^2 \Delta' \big( \tilde{B}^{(\bbe + t\bal)}
    \big)^3}{J \big( \tilde{B}^{(\bbe+ t\bal)} + r U \big)^4}
  \norm{\bal}^2 + \frac{4 n^3 r^4 U^4
    \bar{\alpha}^2}{\big(\tilde{B}^{(\bbe+ t\bal)} + r U \big)^2
    (nrU+z)^2} \\ & + \frac{16 n^3 r^4 U^4 (\bar{\beta} +
    t\bar{\alpha} )^2 \big( \tilde{B}^{(\bbe + t\bal)} \big)^3}{J
    \big(\tilde{B}^{(\bbe+ t\bal)} + r U \big)^4 (nrU+z)^2}
  \norm{\bal}^2 + \frac{16 n r^2 U^2 z^2 (w - \bar{y})^2 \big(
    \tilde{B}^{(\bbe + t\bal)} \big)^3}{J \big(\tilde{B}^{(\bbe+
      t\bal)} + r U \big)^4 (nrU+z)^2} \norm{\bal}^2 \\ & + \frac{2 n
    r^2 U^2 \big( \tilde{B}^{(\bbe + t\bal)} \big)^3}{J
    \big(\tilde{B}^{(\bbe+ t\bal)} + r U \big)^4 (nrU+z)}
  \norm{\bal}^2 + \frac{n \big( \tilde{B}^{(\bbe + t\bal)} \big)^3}{2
    J \big(\tilde{B}^{(\bbe+ t\bal)} + r U \big)^3} \norm{\bal}^2
  \Bigg] \;.
\end{align*}
By Cauchy-Schwarz, $n \bar{\alpha}^2 \le \norm{\bal}^2$, and
\[
n (\bar{\beta} + t\bar{\alpha} )^2 \tilde{B}^{(\bbe + t\bal)} =
\frac{n (\bar{\beta} + t\bar{\alpha} )^2 J}{b + \frac{1}{2}
  \sum_{i=1}^n (\beta_i + t \alpha_i)^2} \le \frac{2 J n (\bar{\beta}
  + t\bar{\alpha} )^2}{\sum_{i=1}^n (\beta_i + t \alpha_i)^2} \le 2 J
\;.
\]
Thus,
\begin{align*}
\mbox{E} \Bigg[ \sum_{i=1}^n \bigg( & \frac{\df}{\df t}
  \tilde{\beta}_{i}^{(\bbe + t\bal)} \bigg)^2 \Bigg] \le \mbox{E}
\Bigg[ \frac{4 r^2 U^2 \Delta' \big( \tilde{B}^{(\bbe + t\bal)}
    \big)^3}{J \big( \tilde{B}^{(\bbe+ t\bal)} + r U \big)^4} +
  \frac{4 n^2 r^4 U^4}{\big(\tilde{B}^{(\bbe+ t\bal)} + r U \big)^2
    (nrU+z)^2} \\ & + \frac{32 n^2 r^4 U^4 \big( \tilde{B}^{(\bbe +
      t\bal)} \big)^2}{\big(\tilde{B}^{(\bbe+ t\bal)} + r U \big)^4
    (nrU+z)^2} + \frac{16 n r^2 U^2 z^2 (w - \bar{y})^2 \big(
    \tilde{B}^{(\bbe + t\bal)} \big)^3}{J \big(\tilde{B}^{(\bbe+
      t\bal)} + r U \big)^4 (nrU+z)^2} \\ & + \frac{2 n r^2 U^2 \big(
    \tilde{B}^{(\bbe + t\bal)} \big)^3}{J \big(\tilde{B}^{(\bbe+
      t\bal)} + r U \big)^4 (nrU+z)} + \frac{n \big( \tilde{B}^{(\bbe
      + t\bal)} \big)^3}{2 J \big(\tilde{B}^{(\bbe+ t\bal)} + r U
    \big)^3} \Bigg] \norm{\bal}^2 \;.
\end{align*}
Note that $\tilde{B}^{(\bbe+ t\bal)} \le J/b$.  Hence,
\begin{align*}
\mbox{E} & \Bigg[ \sum_{i=1}^n \bigg( \frac{\df}{\df t}
  \tilde{\beta}_{i}^{(\bbe + t\bal)} \bigg)^2 \Bigg] \\ & \le \mbox{E}
\Bigg[ \frac{4 \Delta' J^2}{b^3 r^2 U^2} + \frac{4 n^2 r^2 U^2}{z^2}
  + \frac{32 J^2}{b^2 r^2 U^2} + \frac{16 n (w - \bar{y})^2 J^2}{b^3
    r^2 U^2} + \frac{2 J^2}{b^3 r^3 U^3} + \frac{n}{2 b r U} \Bigg]
\norm{\bal}^2 \\ & \le \Bigg[ \frac{(2a+n+2)^2}{4} \bigg(\frac{4
    \Delta'}{b^3 r^2 U^2} + \frac{32}{b^2 r^2 U^2} + \frac{16 n (w -
    \bar{y})^2}{b^3 r^2 U^2} + \frac{2}{b^3 r^3 U^3} \bigg) + \frac{4
    n^2 r^2 U^2}{z^2} + \frac{n}{2 b r U} \Bigg] \norm{\bal}^2 \;.
\end{align*}
Therefore, using \eqref{eq:d3}, we have
\[
 \mbox{E} \, \norm{f(\bbe) - f(\bbe')} \le \underset{t \in
   [0,1]}{\sup} \sqrt{ \mbox{E} \sum_{i=1}^{n} \bigg( \frac{\df}{\df
     t} \tilde{\beta}_i^{(\bbe + t(\bbe' - \bbe))} \bigg)^2} \le
 \gamma_{n,r} \, \norm{\bbe - \bbe'} \;,
\]
where
\begin{align*}
\gamma_{n,r} & = \gamma_{n,r}(\by,U,a,b,w,z) \\ & =
\sqrt{\frac{(2a+n+2)^2}{4} \bigg(\frac{4 \Delta'}{b^3 r^2 U^2} +
  \frac{32}{b^2 r^2 U^2} + \frac{16 n (w - \bar{y})^2}{b^3 r^2 U^2} +
  \frac{2}{b^3 r^3 U^3} \bigg) + \frac{4 n^2 r^2 U^2}{z^2} +
  \frac{n}{2 b r U}} \;.
\end{align*}
Under (A1)-(A4), $\gamma_{n,r} \rightarrow 0$ and $n \rightarrow
\infty$.  Thus, for all large $n$, $\gamma_{n,r}<1$, and
Proposition~\ref{prop:ol} implies that, for every $\bbe \in
\mathbb{R}^n$, we have
\[
d_{\mbox{\scriptsize{W}}}(K^m_{\bbe},\Pi) \le 
\frac{c(\bbe)}{1-\gamma_{n,r}} \, \gamma^m_{n,r} \;,
\]
where $c(\bbe) = c(\bbe;n,r,\by,U,a,b,w,z)$.  The proof is now
complete.
\end{proof}

\bibliographystyle{ims}
\bibliography{bibfile}

\end{document}